\documentclass[12pt,centertags,oneside]{amsart}
\usepackage[T1]{fontenc}
\usepackage[utf8]{inputenc}
\usepackage{amsmath}
\usepackage{amsthm}
\usepackage{amssymb}
\usepackage[mathscr]{eucal}
\usepackage{bbm}
\usepackage{pgf,tikz}
\usepackage[bookmarks,colorlinks=true,linkcolor=cyan,citecolor=blue,urlcolor=cyan]{hyperref}
\usepackage{faktor}
\usepackage{enumitem}
\usepackage{verbatim}
\usepackage{pgfplots}
\usepackage{marginnote}
\usepackage{fullpage}

\usetikzlibrary{arrows}
\usetikzlibrary{shapes.geometric}
\usetikzlibrary{cd}
\usetikzlibrary{intersections, pgfplots.fillbetween}

\newcounter{basicblock}[section]
\renewcommand{\thebasicblock}{\thesection.\arabic{basicblock}}
\newenvironment{block}[1][]
{\refstepcounter{basicblock}\par\medskip\noindent\textbf{#1~\thebasicblock}.}
{\medskip}
\newenvironment{block*}[1][] {\par\medskip\noindent\textbf{#1}} {\medskip}

\newcommand{\ep}{\varepsilon}

\let\oldforall\forall
\renewcommand{\forall}{\oldforall \,}
\let\oldexists\exists
\renewcommand{\exists}{\oldexists \,}

\newcommand{\dist}{\mathrm{dist}}

\newcommand{\di}{\mathrm{d}}

\newcommand{\ddc}{\mathrm{dd^c}}
\newcommand{\de}{\partial}

\newcommand{\fs}{\omega_{FS}}

\newcommand{\cC}{\mathcal{C}}
\newcommand{\cD}{\mathcal{D}}

\newcommand{\cL}{\mathcal{L}}

\newcommand{\cU}{\mathcal{U}}

\newcommand{\B}{\mathbb{B}}
\newcommand{\C}{\mathbb{C}}

\newcommand{\N}{\mathbb{N}}
\newcommand{\Pb}{\mathbb{P}}
\newcommand{\R}{\mathbb{R}}

\setlist[itemize]{noitemsep, topsep=3pt}
\setlist[enumerate]{noitemsep, topsep=3pt}
\newenvironment{nlist}
    {\begin{enumerate}[label=(\roman*)]}
    {\end{enumerate}}

\newenvironment{defn}{\begin{block}[Definition]}{\end{block}}
\newenvironment{lm}{\begin{block}[Lemma]\begin{itshape}}{\end{itshape}\end{block}}
\newenvironment{prop}{\begin{block}[Proposition]\begin{itshape}}{\end{itshape}\end{block}}
\newenvironment{thm}{\begin{block}[Theorem]\begin{itshape}}{\end{itshape}\end{block}}
\newenvironment{cor}{\begin{block}[Corollary]\begin{itshape}}{\end{itshape}\end{block}}

\newenvironment{rmk}{\begin{block}[Remark]}{\end{block}}

\newenvironment{notn}{\begin{block*}[Notations.]}{\end{block*}}
\newenvironment{outline}{\begin{block*}[Outline of the paper.]}{\end{block*}}
\newenvironment{idea}{\begin{block*}[Main idea.]}{\end{block*}}

\title{log-H\"older regularity of currents and \\equidistribution towards Green currents}
\author{Marco Vergamini}
\address{Scuola Normale Superiore, Pisa, Italy}
\email{marco.vergamini@sns.it}
\date{}

\begin{document}

\begin{abstract}
    Let $f$ be an endomorphism of a projective space or an automorphism of a compact K\"ahler manifold. We prove that the pull-backs of currents under the iterates of $f$ converge exponentially fast to the Green currents when tested at $\log$-H\"older-continuous observables whose $\ddc$'s have bounded mass.
\end{abstract}

\maketitle

\medskip


\smallskip


\section{Introduction} \label{intro}
It is a now classical result that every rational map on the Riemann sphere of degree at least $2$ admits a unique measure of maximal entropy, called the \emph{equilibrium measure}, see \cite{B65AM,FLM83BSBM,L83ETDS}, and preimages of generic points equidistribute towards it. These results have been extended to the case of endomorphisms of projective spaces in any dimension and for currents of any bidegree, see for instance \cite{DS09AM,FS94NATOASI}. Dinh and Sibony also proved that the rate of the equidistribution towards the equilibrium measure is exponential \cite{DS10MA}, see also \cite{T11AM} for the case of bidegree $(1,1)$.

Another natural and interesting class of holomorphic dynamical systems is given by the automorphisms of compact K\"ahler manifolds. In this case, under suitable assumptions, we know that there are two invariant currents $T_+$ and $T_-$, the \emph{Green currents} associated to the automorphism and its inverse, which have H\"older-continuous \emph{super-potentials}, see \cite{DS10JAG}. In the same paper, Dinh and Sibony proved the equidistribution towards these currents for forms of suitable degree. Equidistribution properties on K\"ahler manifolds have also been studied for correspondences, which are a natural generalization of both endomorphisms of projective spaces and automorphisms of K\"ahler manifolds, see \cite{DNV18AM,AV24AMP}.

\smallskip

In the results above, the speed of equidistribution has been given in the case where the currents are evaluated at H\"older-continuous test functions or forms. Recently, Bianchi and Dinh showed in \cite{BD23JMPA} the importance of $\log$-H\"older regularity, a condition which is ``exponentially'' weaker than being H\"older, in the study of equilibrium states for projective spaces. They further developed the study of this class of functions and applied it to obtain several statistical properties of equilibrium states, see \cite{BD24GAFA}.

The $\log$-H\"older regularity has also a very nice property that makes it adapted to holomorphic dynamics. Like in the classical H\"older case, the $\log$-H\"older condition also has an \textit{exponent of regularity}. But while the exponent of regularity of H\"older-continuous functions in general decreases under push-forward of holomorphic functions, and so there is a loss of regularity, the one of $\log$-H\"older-continuous functions stays the same. In particular, functions of this regularity are naturally suited for studying the dynamics.

Unfortunately, this is not the case in higher bidegree, where the push-forward of even a smooth form may not be continuous. To extend the study of $\log$-H\"older regularity to any bidegree, we then move our attention from the regularity of forms to that of currents, relying on the theory of \emph{super-potentials} developed by Dinh and Sibony, see for instance \cite{DS10JAG}. We introduce the notion of $\log$-H\"older-continuous super-potentials and we give a precise control on this regularity under iterations of holomorphic functions, see for instance Proposition \ref{pushforward_kahler} and Corollary \ref{pushforward_logq}. In particular, we study the action of both endomorphisms of projective spaces and automorphisms of compact K\"ahler manifolds on currents of this regularity, and we use it to extend the quantitative equidistribution results by Dinh and Sibony to $\log$-H\"older-continuous test forms.

\smallskip

Many applications of equidistribution speed can be found in the literature. One can use it to derive statistical results for the equilibrium measures such as mixing and central limit theorem, see for instance \cite{BD26AM,DTV24Ax,VW25Ax}. There are also other applications: for example, in \cite{BBR25AMP} the authors used the speed of equidistribution to infer geometrical properties of measures of large entropy.

\medskip

Our first main result concerns the rate of equidistribution of currents on K\"ahler manifolds when evaluated at $\log$-H\"older-continuous test forms. This result will be a consequence of the speed of convergence of $\log$-H\"older-continuous super-potentials, see Theorem \ref{ennequsuqupiuuno}. We first need some definitions.

\smallskip

Let $(X,\omega)$ be a compact K\"ahler manifold of dimension $k$, and let $f$ be a holomorphic automorphism of $X$. The \textit{dynamical degree of order $s$} of $f$ is the spectral radius of the pull-back operator $f^*$ acting on the Hodge cohomology group $H^{s,s}(X,\mathbb{R})$. It is denoted by $d_s(f)$, or simply by $d_s$ if there is no confusion. We have $d_0 = d_k = 1$ and $d_s(f^n) = d_s^n$.

A theorem by Khovanskii \cite{K79UMN}, Teissier \cite{T79CRASP}, and Gromov \cite{G90ADGT} implies that the sequence ${s\mapsto\log{d_s}}$ is concave. So, there are integers $0\le p\le p'\le k$ such that
$$1=d_0<\dots<d_p=\dots=d_{p'}>\dots>d_k=1.$$

We assume that $f$ has \emph{simple action on cohomology}, i.e., that we have ${p=p'}$ and $f^*$, acting on $H^{p,p}(X,\mathbb{R})$, admits only one eigenvalue of maximal modulus $d_p$. This are natural condition to ask for when studying the dynamics of automorphisms of K\"ahler manifolds, see for instance \cite[Section 4.4]{DS10JAG} and \cite{DS10CM}. We fix constants $\max\{d_{p-1},d_{p+1}\}<\delta'<\delta<d_p$ such that all the eigenvalues of $f^*$ acting on $H^{p,p}(X,\mathbb{R})$, except for $d_p$, are smaller than $\delta'$. We call $d_p$ the \emph{main dynamical degree} and $\delta$ the \emph{auxiliary dynamical degree} of $f$. Every dependence of the constants on $\delta,\delta'$ will be implicitly contained in the dependences on $f$.

From \cite{DS05JAMS, DS10JAG} we know that there exist a positive closed $(p, p)$-current $T_+$ and a positive closed $(k-p, k-p)$-current $T_-$, the \emph{Green currents} of $f$ and $f^{-1}$ respectively, such that $f^*(T_+)=d_pT_+$ and ${f_*(T_-)=d_pT_-}$.

\medskip

This is our first result, see to Definitions \ref{holdlog} and \ref{def-logq-form} for the definition of $\log^q$-continuous forms and of the norm $\|\cdot\|_q$ for $\log^q$-continuous forms with bounded mass of the $\ddc$. This norm plays an important role in holomorphic dynamics, see for instance \cite[Section 3.2]{BD24GAFA}.

We denote by $\cD_p$ the space generated by positive closed currents on $X$, and by $\cD_p^1$ its subset containing currents with mass bounded by $1$.

\begin{thm} \label{goal_kahler}
    Let $f$ be an automorphism of a compact K\"ahler manifold $X$ with simple action on cohomology. Let $T_+$ be its Green current, $d_p$ its main dynamical degree, and $\delta$ its auxiliary dynamical degree. For every $q>0$ and for every $\log^q$-continuous $(k-p,k-p)$-form $\Phi$ with $\|\ddc\Phi\|_*$ bounded we have
    \begin{equation} \label{conv-kahler}
        \left|\left\langle \frac{(f^*)^n}{d_p^n}(S)-rT_+,\Phi\right\rangle\right| \lesssim \|\Phi\|_q\left(\dfrac{\delta}{d_p}\right)^{\frac{nq}{q+1}}
    \end{equation}
    for every current $S$ in $\mathcal{D}_p^1$, where $r$ is such that $d_p^{-n}(f^n)^*(S)$ converge to $rT_+$. The implicit constant depends only on $f$, $q$ and the cohomology class of $S$.
\end{thm}

Our second main result is a version of Theorem \ref{goal_kahler} for endomorphisms of projective spaces. As in the previous case, the result will be a consequence of the speed of convergence of $\log$-H\"older-continuous super-potentials, see Theorem \ref{logq_esp_pot}. In this case we have a \emph{Green $(1,1)$-current} $T$ such that in any bidegree $(s,s)$ the corresponding Green current is given by $T^s$, and we have $f^*(T^s)=d^sT^s$ where $d$ is the algebraic degree of the endomorphism $f$. Notice that in the projective case there is not a special bidegree $(p,p)$, but the result holds for every bidegree. The case when $s=k$ and $S$ is the Dirac mass at any point of $\mathbb{P}^k$ was proved in \cite[Theorem 3.2]{BD24GAFA}.

\begin{thm} \label{logq_esp_test}
    Let $f$ be a generic endomorphism of $\mathbb{P}^k$. There exists $0<\kappa<d$ such that for every $0\le s\le k$, for every $q>0$ and for all $\log^q$-continuous $(k-s,k-s)$-forms $\Phi$ with $\|\ddc\Phi\|_*$ bounded, we have
    \begin{equation} \label{pk-end-1}
        \left|\left\langle \frac{(f^*)^n}{d^{sn}}(S)-T^s,\Phi\right\rangle\right| \lesssim \|\Phi\|_q\left(\dfrac{\kappa}{d}\right)^{\frac{nq}{q+1}}
    \end{equation}
    for every positive closed $(s,s)$-current $S$ of mass $1$, where the implicit constant depends only on $f$ and $q$.
\end{thm}

We will discuss in Section \ref{kahler_sub2} some statistical applications of our main results, namely exponential mixing of all orders and central limit theorem for $\log$-continuous observables, see Theorem \ref{expmix}.

\medskip

The techniques developed in this paper are also naturally adapted to the setting of holomorphic correspondences on compact K\"ahler manifolds \cite{DNV18AM,DS06CMH}. Here, one necessarily faces difficulties in the study of both $f^*$ and $f_*$ acting on currents, similar to those of $f^*$ in the case of endomorphisms of $\Pb^k$ (see in particular Lemma \ref{pullback_holder}). These problems have been addressed in \cite{LV26Ax}, and the $\log$-H\"older-regularity introduced in the present paper turns out to be naturally adapted to this setting: the super-potentials of Green currents of holomorphic correspondences have this regularity.

It is a natural question whether the methods introduced in this paper can be extended to the case of horizontal-like maps in any dimension \cite{DNS08AM}. Here, the non-compactness of the space makes it not possible to apply some crucial steps of our construction.

\begin{idea}
    In both Theorems \ref{goal_kahler} and \ref{logq_esp_test}, we reduce ourselves to study the action of $f^*$ on the space of exact currents. This will translate, by duality, on the study of the action of $f_*$ on the space of exact currents with $\log$-H\"older-continuous super-potentials.
    
    Our main theorems will then be the consequences of Theorems \ref{ennequsuqupiuuno} and \ref{logq_esp_pot}, which give the speed of convergence of super-potentials of this regularity. To prove these theorems, we need various results about currents with regular super-potentials. This will be the main focus of Section \ref{techs}, where we develop the theory of $\log$-H\"older-regularity both for forms and currents. In particular, we prove a ``Skoda-type'' estimate for currents with $\log$-H\"older-continuous super-potentials, see Proposition \ref{skoda_clicks}.
    
    Then, we need to prove that the push-forward of currents preserves the $\log$-H\"older-continuity of super-potentials, with a precise control on the multiplicative constant. For automorphisms of compact K\"ahler manifolds this follows without difficulties from the invertibility of $f$. For projective spaces the proof is more delicate, and the non-invertibility of endomorphisms will be one of the main difficulties to handle in this case, see Lemma \ref{pullback_holder}.
\end{idea}

\begin{outline}
    In Section \ref{prels} we recall some notions about currents, in particular the one of super-potentials. In Section \ref{techs} we define the $\log$-H\"older regularity and we prove some useful results about currents with $\log$-H\"older-continuous super-potentials. In Section \ref{kahler} we study the case of automorphisms of compact K\"ahler manifolds; in Section \ref{kahler_sub1} we prove Theorem \ref{goal_kahler}, while in Section \ref{kahler_sub2} we derive exponential mixing as an application. In Section \ref{endo} we deal with the more delicate case of endomorphisms of projective spaces and we prove Theorem \ref{logq_esp_test}.
\end{outline}

\begin{notn}
    Through all the article, $(X,\omega)$ is a fixed compact K\"ahler manifold of dimension $k$. In Section \ref{endo} we will take specifically $X=\mathbb{P}^k$ and $\omega=\omega_{FS}$ the Fubini-Study form. The symbols $\lesssim$ and $\gtrsim$ stand for inequalities up to a positive multiplicative constant. Given two real currents of the same bidegree $S_1$ and $S_2$, we write $S_1\le S_2$ or $S_2\ge S_1$ if $S_2-S_1$ is positive. We also write $|S_1|\le S_2$ to mean $-S_2\le S_1\le S_2$. Given a positive or negative $(s,s)$-current $S$, we define the mass of $S$ as $\|S\|:=|\langle S,\omega^{k-s}\rangle|$.
\end{notn}

\bigskip

\noindent\textbf{Fundings and Acknowledgements.} The author is part of the PHC Galileo project G24-123.

\medskip

\noindent\textbf{Statements and Declarations.} The author has no conflicts of interest to declare that are relevant to the content of this article.

\section{Preliminaries on currents}\label{prels}
In this section we will give some general definitions and results which will be used through all the article.

\subsection{The space of currents}\label{space_of_currs}

Let $X$ be a compact K\"ahler manifold of dimension $k$. Denote by $\mathcal{D}_s(X)$ the real space generated by the positive closed $(s, s)$-currents on $X$, and by $\mathcal{D}_s^0(X)$ the subspace of $\mathcal{D}_s(X)$ given by exact currents, i.e., currents $S$ whose class $\{S\}\in H^{s,s}(X,\mathbb{R})$ is zero. As $X$ will be fixed, we will omit the dependence of these spaces on it.

Given a sequence of currents $(S_n)_{n\ge0}$ and a current $S$, all of the same bidegree, we say that \textit{$S_n$ converge to $S$ in the sense of currents}, and we write $S_n\rightharpoonup S$, if they converge pointwise as linear forms.  We call \textit{weak topology} the topology induced by convergence in the sense of currents.

We now recall some notions about currents on $X$, see \cite[Section 2.2]{DS10JAG}.

Let $S$ be a current in $\mathcal{D}_s$. We recall the norm
$$\|S\|_*:=\min\{\|S^+\|+\|S^-\|\},$$
where the minimum is taken over all positive closed $(s,s)$-currents $S^\pm$ such that $S=S^+-S^-$.

\smallskip

We define the following topology on $\mathcal{D}_s$: given a sequence of currents $(S_n)_{n\ge0}$ and a current $S$, we say that the $S_n$'s \emph{$*$-converge} to $S$ (or simply ``converge'') if they converge in the sense of currents and they are $\|\cdot\|_*$-bounded. We call this topology the \emph{$*$-topology}. Smooth forms are dense in $\mathcal{D}_s$ and $\mathcal{D}^0_s$ for the $*$-topology, see \cite[Theorem 2.4.4]{DS10JAG}.

\medskip

Since $X$ is compact, we can fix a finite atlas. For every $l\ge0$, we consider the standard $\|\cdot\|_{\mathcal{C}^l}$ norm of forms associated to that atlas. Different choices of atlases give equivalent norms. Therefore, we can fix an atlas, and hereafter every dependence on it will be omitted. When talking about forms with bounded measurable coefficients, we will use the norm $\|\cdot\|_\infty$, which coincides with $\|\cdot\|_{\mathcal{C}^0}$ for continuous forms.

By duality, we can consider the following norms on $\mathcal{D}_s$.

\begin{defn}
    Given a current $S$ in $\mathcal{D}_s$ and $l>0$, we define
    $$\|S\|_{-l}:=\sup\{|\langle S,\Omega\rangle|\mid \Omega\text{ is a smooth }(k-s,k-s)\text{-form with }\|\Omega\|_{\mathcal{C}^l}\le1\}.$$

    The norm $\|\cdot\|_{-l}$ induces a distance $\dist_l$ given by $\dist_l(S,S'):=\|S-S'\|_{-l}$.
\end{defn}

The following result is obtained using a standard interpolation between Banach spaces \cite{T78Book}, see for instance \cite[Proposition 2.2.1]{DS10JAG}.

\begin{prop}\label{interpol}
    Let $l$ and $l'$ be real numbers with $0<l<l'$. Then, on any $\|\cdot\|_*$-bounded subset of $\mathcal{D}_p$, the topologies induced by $\dist_l$ and by $\dist_{l'}$ coincide with the weak topology and the $*$-topology. Moreover, for every $\|\cdot\|_*$-bounded subset of $\mathcal{D}_p$, there is a constant $c_{l,l'} > 0$ such that
    $$\dist_{l'} \le \dist_l \le c_{l,l'}(\dist_{l'})^{l/l'}.$$
\end{prop}
\vspace*{-\baselineskip}

In the sequel, we will need the following result.

\begin{lm}\label{star_lip}
    Let $\tau:X_1\rightarrow X_2$ be a holomorphic map between two compact K\"ahler manifolds of dimensions $k_1$ and $k_2$, respectively. Then the map $\tau_*:\cD_{k_1-s}(X_1)\rightarrow \cD_{k_2-s}(X_2)$ is Lipschitz with respect to $\dist_2$ for every $0\le s\le \min\{k_1,k_2\}$.
\end{lm}

\begin{proof}
    Let $S$ be a current in $\mathcal{D}_{k_1-s}(X_1)$. We need to show that we have $\|\tau_*(S)\|_{-2}\le C\|S\|_{-2}$ for some constant $C>0$ depending only on $\tau$. Fix a smooth $(s,s)$-form $\Omega$ on $X$ with $\|\Omega\|_{\mathcal{C}^2}\le1$. It suffices to show that we have $|\langle \tau_*(S),\Omega\rangle|\le C\|S\|_{-2}$. We have
    $$|\langle \tau_*(S),\Omega\rangle|=|\langle S,\tau^*(\Omega)\rangle| \le \|S\|_{-2}\|\tau^*(\Omega)\|_{\mathcal{C}^2}.$$
    Since $\tau$ is smooth and $\|\Omega\|_{\mathcal{C}^2}\le1$, it follows that $\|\tau^*(\Omega)\|_{\mathcal{C}^2}$ is bounded by a constant that depends only on $\tau$. This completes the proof.
\end{proof}

\subsection{Plurisubharmonic and d.s.h.\ functions} \label{psh_and_dsh}

Take an open subset $U$ of $X$. We recall that a function $u:U\longrightarrow\mathbb{R} \cup \{-\infty\}$, not identically $-\infty$ on any connected component of $U$, is called \textit{plurisubharmonic}  (\textit{p.s.h.}\ for short) if it is upper semi-continuous and, for every holomorphic function ${\tau:\mathbb{D}\longrightarrow U}$, we have that $u\circ\tau$ is subharmonic or identically $-\infty$ on $\mathbb{D}$. A function $u:X\longrightarrow\mathbb{R} \cup \{-\infty\}$ is called \textit{quasi-plurisubharmo\-nic}  (\textit{quasi-p.s.h.}\ for short) if it is locally the difference of a p.s.h.\ function and a smooth one. The space of quasi-p.s.h.\ functions is larger than that of p.s.h.\ functions, which on a compact manifold are necessarily constant by the maximum principle. On the other hand, quasi-p.s.h.\ functions still preserve the main compactness properties of p.s.h.\ functions.

\smallskip

A function $u:X\longrightarrow\mathbb{R} \cup\{\pm \infty\}$ is \textit{d.s.h.}\ if it is the difference of two quasi-p.s.h.\ functions outside of a pluripolar set. Since their introduction in \cite{DS06CPAM,DS06CMH}, d.s.h.\ functions have played a central role in holomorphic dynamics. We denote by $\text{DSH}(X)$ the space of d.s.h.\ functions on $X$. If $u$ is d.s.h., there are two positive closed $(1,1)$-currents $R^\pm$ on $X$ such that $\ddc u=R^+-R^-$. As these two currents are cohomologous, they have the same mass. We define a norm on $\text{DSH}(X)$ by 
$$\|u\|_\text{DSH}:= |\langle\omega^k,u\rangle|+\frac{\|\ddc u\|_*}{2}.$$

Notice that $u\mapsto\|\ddc u\|_*$ is already a seminorm for functions. We obtain a norm equivalent to $\|\cdot\|_\text{DSH}$ if $|\langle\omega^k,u\rangle|$ is replaced by $\|u\|_{L^1(\omega^k)}$ or by $|\langle\nu,u\rangle|$, where $\nu$ is any \textit{PB} measure, i.e., such that all d.s.h.\ functions are integrable with respect to $\nu$.

\subsection{Superpotentials of currents}\label{superpots}

We now recall the definition of super-potentials of currents and review some basic facts about them. We refer to \cite[Section 2]{DS10CM} and \cite[Section 3.2]{DS10JAG}, see also \cite[Section 2]{BD26AM}.

\medskip

Given a current $S\in\mathcal{D}_{k-s+1}^0$, $0<s\le k$, there exists a real $(k-s,k-s)$-current $U_S$ such that $\ddc U_S=S$. We can choose $U_S$ smooth if $S$ is smooth. We can also \textit{normalize} $U_S$ as follows: for every $0\le s\le k$, we fix a family $(\alpha_{s,1},\dots,\alpha_{s,h_s})$ of real smooth closed $(s,s)$-forms, with $h_s:=\dim{H^{s,s}(X,\mathbb{R})}$, such that the family of the cohomology classes $\{\alpha_{s,1}\},\dots,\{\alpha_{s,h_s}\}$ are a basis of $H^{s,s}(X,\mathbb{R})$. From now on, every dependence on the choice of the $\alpha_{j,l}$'s will be omitted. We can add to $U_S$ a suitable real smooth closed form so that $\langle U_S, \alpha_{s,l}\rangle=0$ for every $1\le l\le h_s$. In this case, we say that $U_S$ is \textit{normalized}. Such a current $U_S$ is called a \textit{potential} of $S$, and we will always have it normalized unless otherwise stated. Observe that, once the $\alpha_{j,l}$'s are fixed, a canonical way to obtain a normalization is to use Poincaré duality, see for instance \cite{DS10JAG}.

\begin{defn}
    Take currents $R\in\mathcal{D}_s$ and $S\in\mathcal{D}_{k-s+1}^0$. If either $R$ or $S$ is smooth, define the \textit{(normalized) super-potential of $R$} evaluated at $S$ as
    $$\mathcal{U}_R(S):=\langle R,U_S\rangle.$$
\end{defn}
\vspace*{-\baselineskip}

The value of the super-potential of every current in $\mathcal{D}_s$ is always defined at every smooth current ${S\in\mathcal{D}_{k-s+1}^0}$, and does not depend on the choice of $U_S$  \cite[Lemma 3.2.1]{DS10JAG}. Thus, the super-potential is a linear functional on the space of smooth exact currents.

\begin{defn}
    Take a current $R\in\mathcal{D}_s$. We say that $R$ has \textit{continuous super-potential} if $\mathcal{U}_R$ extends to a linear functional defined on all of $\mathcal{D}_{k-s+1}^0$ which is continuous with respect to the $*$-topology.
\end{defn}

Smooth currents always have continuous super-potentials, see \cite[Section 3.4]{DS10JAG} and Lemma \ref{supofsmooth} for a more precise version of this fact.

\smallskip

We highlight the following property of super-potentials. Although elementary, we will use it several times in the sequel.

\begin{lm} \label{pb-commute}
    Let $f:X\longrightarrow X$ be a holomorphic function. Let $R$ be a current in $\mathcal{D}^0_{k-s+1}$ and $S$ a smooth current in $\mathcal{D}^0_s$. Then we have
    $$\mathcal{U}_{f_*(R)}(S)=\mathcal{U}_R\big(f^*(S)\big).$$
    
    If $\mathcal{U}_R$ is continuous, then the equality extends to all currents $S$ in $\mathcal{D}^0_s$.
\end{lm}

\begin{proof}
    We have
    $$\mathcal{U}_{f_*(R)}(S)=\left\langle f_*(R), U_S\right\rangle=\left\langle R, f^*(U_S)\right\rangle.$$
    Since $R$ is in $\mathcal{D}^0_{k-s+1}$ and we have $\ddc f^*(U_S)=f^*(\ddc U_S)=f^*(S)=\ddc U_{f^*(S)}$, we get
    $$\left\langle R, f^*(U_S)\right\rangle=\left\langle R, U_{f^*(S)}\right\rangle=\mathcal{U}_R\big(f^*(S)\big),$$
    which proves the first assertion.

    Since $f^*$ is continuous, if $\mathcal{U}_R$ is continuous the equality above extends to all currents $S$ in $\mathcal{D}^0_s$. This proves the second assertion and completes the proof.
\end{proof}

In the sequel, we will also make use of the following result, see for instance \cite[Section 2]{DS10CM}.

\begin{lm} \label{shuffle}
    Let $R$ be a current in $\mathcal{D}^0_{k-s+1}$ and $S$ a smooth current in $\mathcal{D}^0_s$. We have
    $$\cU_R(S)=\cU_S(R).$$

    If $\mathcal{U}_R$ is continuous, then the equality extends to all currents $S$ in $\mathcal{D}^0_s$.
\end{lm}

\section{Currents with log-H\"older-continuous super-potentials}\label{techs}
In this section we introduce the notion of $\log$-H\"older-continuous super-potential and we study the space of currents with this property.

\subsection{Regular super-potentials} \label{reg_superpots}

In the sequel we will work with super-potentials and forms which are more than just continuous. We now define the two main type of regularities that we are interested in: \textit{H\"older} and \textit{$\log$-H\"older}.

Given a monotone and continuous function $w:[0,1)\longrightarrow[0,+\infty)$, called \textit{modulus of continuity}, we say that a function $f$ between two compact metric spaces $(Y_1,d_1)$ and $(Y_2,d_2)$ is \textit{$w$-continuous} if it admits $w$ as modulus of continuity, i.e., if we have
$$d_2\big(f(x),f(y)\big)\le w\big(d_1(x,y)\big)\quad\text{for every}\quad x,y\in Y_1\quad\text{with}\quad d_1(x,y)<1.$$

When we consider a function whose domain is (a subset of) $Y_1=\mathcal{D}_s^1:=\mathcal{D}_s\cap\{\|S\|_*\le1\}$, the modulus of continuity will always be used in the case $d_1=\dist_2$.

\begin{defn} \label{holdlog}
    For $C,\lambda,M,q>0$, put
    $$w^{C,\lambda}(\delta):=C\delta^\lambda\qquad\text{and}\qquad w_{M,q}(\delta):=\dfrac{M}{(1+|\log{\delta}|)^q}.$$
    We say that a function between two metric spaces is:
    \begin{enumerate}[label={(\arabic*)}]
        \item \textit{$(C,\lambda)$-H\"older-continuous} if it is $w^{C,\lambda}$-continuous;
        \item \textit{$\lambda$-H\"older-continuous} if it is $(C,\lambda)$-H\"older-continuous for some $C>0$;
        \item \textit{H\"older-continuous} if it is $\lambda$-H\"older-continuous for some $\lambda>0$;
        \item \textit{$(M,q)$-$\log$-H\"older-continuous} (or \textit{$(M,q)$-$\log$-continuous}) if it is $w_{M,q}$-continuous;
        \item \textit{$q$-$\log$-H\"older-continuous} (or \textit{$\log^q$-continuous}) if it is $(M,q)$-$\log$-continuous for some $M>0$;
        \item \textit{$\log$-H\"older-continuous} (or \textit{$\log$-continuous}) if it is $q$-$\log$-continuous for some $q>0$.
    \end{enumerate}

    When considering a form $\Phi$, we say that it is $(C,\lambda)$-H\"older-continuous if all its coefficients in all the charts are $(C,\lambda)$-H\"older-continuous. We define analogously $(M,q)$-$\log$-H\"older-continuous forms.
\end{defn}

For $\log^q$-continuous forms, we have the following norm.

\begin{defn}
    Given a $\log^q$-continuous form $\Phi$ on $X$, we define
    $$\|\Phi\|_{\log^q}:=\|\Phi\|_\infty+\inf\{M>0\mid\Phi\text{ is }(M,q)\text{-}\log\text{-continuous}\}.$$
\end{defn}
\vspace*{-\baselineskip}

The following norm is defined in \cite[Section 3.2]{BD24GAFA} in the case of functions. We extend it to forms of any bidegree.

\begin{defn}\label{def-logq-form}
    Given a $\log^q$-continuous form $\Phi$ on $X$ with $\|\ddc\Phi\|_*$ bounded, we define
    $$\|\Phi\|_q:=\|\Phi\|_{\log^q}+\|\ddc\Phi\|_*.$$
\end{defn}
\vspace*{-\baselineskip}

Observe that, since $\log^q$-continuity is preserved under changes of charts, up to norm equivalence the definitions of $\|\cdot\|_{\log^q}$ and $\|\cdot\|_q$ are independent of the choice of atlas.

\begin{defn}
    Take $R\in\mathcal{D}_s$. Given a modulus of continuity $w$, we say that $\mathcal{U}_R$ is \textit{$w$-continuous} if it is continuous and we have
    $$|\mathcal{U}_R(S)|\le w(\|S\|_{-2})\qquad\text{for every }S\in\mathcal{D}^0_{k-s+1}\cap \mathcal{D}^1_{k-s+1}.$$
\end{defn}
\vspace*{-\baselineskip}

\begin{rmk} \label{dense}
    Recall that smooth forms are dense in $\mathcal{D}_{k-s+1}^0$ for the $*$-topology. So, since a uniformly continuous function can be uniquely extended as a continuous function to the closure of its domain and the extension still has the same modulus of continuity, one can test for $w$-continuity of super-potentials just on smooth forms. This will automatically give the $w$-continuous extension to all of $\mathcal{D}_{k-s+1}^0\cap\mathcal{D}_{k-s+1}^1$.
\end{rmk}

For smooth forms we have the following result, see for instance \cite[Section 3.4]{DS10JAG}, and \cite[Lemma 2.8 (i)]{BD26AM} for a detailed proof.

\begin{lm} \label{supofsmooth}
    Let $\Phi$ be a smooth $(s,s)$-form on $X$. Then its super-potential is $(C,1)$-H\"older-continuous for some $0<C\lesssim\|\Phi\|_{\cC^2}$, where the implicit constant depends only on the choice made for the normalization of potentials.
\end{lm}

We also have the following result about the super-potential of the $\ddc$ of a H\"older-continuous form.

\begin{lm} \label{hold_sup_hold}
    Let $\Phi$ be a $\lambda$-H\"older-continuous $(s,s)$-form for some $\lambda>0$. The super-potential of $\ddc\Phi$ is $(C,\min\{1,\lambda/2\})$-H\"older-continuous for some $0<C\lesssim\|\Phi\|_{\cC^\lambda}$, where the implicit constant may depend on $\lambda$, but is independent of $\Phi$.

    Moreover, for every current $S\in\cD^0_{k-s}\cap\cD^1_{k-s}$ the value of $\cU_{\ddc\Phi}(S)$ is given by $\langle S,\Phi\rangle$.
\end{lm}

\begin{proof}
    Take a current $S$ in $\cD^0_{k-s}\cap\cD^1_{k-s}$. By definition, we have $|\langle S,\Phi\rangle|\le \|\Phi\|_{\cC^\lambda}\|S\|_{-\lambda}$. Applying Proposition \ref{interpol}, we get
    $$|\langle S,\Phi\rangle|\lesssim \|\Phi\|_{\cC^\lambda}\|S\|_{-2}^{\min\{1,\lambda/2\}},$$
    where the implicit constant depends only on $\lambda$. In order to conclude, it remains to prove that $\cU_{\ddc\Phi}(S)$ is indeed equal to $\langle S,\Phi\rangle$. By Remark \ref{dense}, we just need to prove it for $S$ smooth.

    Take $S$ smooth, and let $U_S$ be a smooth normalized potential of $S$. We have
    $$\cU_{\ddc\Phi}(S)=\langle \ddc\Phi, U_S\rangle=\langle \Phi,\ddc U_S\rangle=\langle \Phi, S\rangle.$$
    The proof is complete.
\end{proof}

We now study the regularity of the super-potential of the $\ddc$ of a $\log$-H\"older-continuous form.

\begin{lm} \label{logq_test}
    Let $\Phi$ be a $\log^q$-continuous $(s,s)$-form on $X$ for some $q>0$. Then the function
    \begin{equation}\label{acting_on_currents}
        S \longmapsto |\langle S,\Phi\rangle|
    \end{equation}
    is $(M,q)$-$\log$-continuous in $\mathcal{D}_{k-s}^1$ for some $0<M\lesssim \|\Phi\|_{\log^q}$, where the implicit constant may depend on $q$, but is independent of $\Phi$.
\end{lm}

\begin{proof}
    Recall that we fixed an atlas on $X$. Let $\varphi_1,\dots,\varphi_N$ be a partition of unity associated to that atlas. Then we have $\langle S,\Phi\rangle=\displaystyle\sum_{j=1}^N\langle S,\Phi\varphi_j\rangle$, and the fact that $\varphi_j$ is smooth for every $j$ implies that $\Phi\varphi_j$ is $(M',q)$-$\log$-continuous for some $0<M'\lesssim \|\Phi\|_{\log^q}$, where the implicit constant depends only on $q$. We also have
    $$|\langle S,\Phi\rangle|\le\sum_{j=1}^N|\langle S,\Phi\varphi_j\rangle|.$$
    It then suffices to estimate $|\langle S,\Phi\varphi\rangle|$ for one fixed $\varphi=\varphi_j$ with compact support in a chart $U$, which we identify as an open subset of $\mathbb{C}^k$. Put $\Omega=\Phi\varphi$ and, for every $z\in\C^k$ and $\theta\in\C$ with $|\theta|\le1$, define
    $$\Omega_\theta(z):=\int_{a\in\mathbb{B}}\Omega(z+r\theta a)\rho(a)\di a,$$
    where $\mathbb{B}$ is the unit ball of $\mathbb{C}^k$ and $\rho$ is a fixed smooth function with compact support in $\mathbb{B}$ such that $\int\rho(a)\di a=1$. The integral above is computed separately for each coefficient of $\Omega$. We can choose $r<1$ sufficiently small, so that for all $|\theta|\le1$ the support of $\Omega_\theta$ is contained in a compact subset of $U$. Hence, we can extend $\Omega_\theta$ to $0$ on $X\setminus U$.\\
    
    \textbf{Claim 1.} We have
    $$\|\Omega_\theta\|_{\mathcal{C}^2}\lesssim \|\Omega_\theta\|_{\mathcal{C}^2(U)}\lesssim \|\Omega\|_{\mathcal{C}^0}|\theta|^{-2}\quad\text{and}\quad\|\Omega\|_{\mathcal{C}^0}\lesssim\|\Phi\|_\infty \le\|\Phi\|_{\log^q}.$$
    \begin{proof}
        The second chain of inequalities is clear. For the first one, the first inequality follows from a change of variables  between charts, and the implicit constant depends only on $X$. For the second inequality, let $z=(x_1+iy_1,\dots,x_k+iy_k)$ be the coordinates in $\C^k$. We observe that for every $j_1,l_1,\dots,j_k,l_k\in\N$ we have
        \begin{align}\label{changov}
            \frac{\de^{j_1+l_1+\dots+j_k+l_k}}{\de_{x_1}^{j_1}\de_{y_1}^{l_1}\dots\de_{x_k}^{j_k}\de_{y_k}^{l_k}}&\Omega_\theta(z)\\ \nonumber
            &=\frac{\de^{j_1+l_1+\dots+j_k+l_k}}{\de_{x_1}^{j_1}\de_{y_1}^{l_1}\dots\de_{x_k}^{j_k}\de_{y_k}^{l_k}}\int_{a\in\C^k} \Omega(z+r\theta a)\rho(a)\di a\\ \nonumber
            &=\frac{\de^{j_1+l_1+\dots+j_k+l_k}}{\de_{x_1}^{j_1}\de_{y_1}^{l_1}\dots\de_{x_k}^{j_k}\de_{y_k}^{l_k}}|r\theta|^{-2k}\int_{w\in\C^k} \Omega(w)\rho\left(\frac{w-z}{r\theta}\right)\di w\\ \nonumber
            &=|r\theta|^{-2k-(j_1+l_1+\dots+j_k+l_k)}\int_{w\in\C^k} \Omega(w)\left(\frac{\de^{j_1+l_1+\dots+j_k+l_k}}{\de_{x_1}^{j_1}\de_{y_1}^{l_1}\dots\de_{x_k}^{j_k}\de_{y_k}^{l_k}}\rho\right)\left(\frac{w-z}{r\theta}\right)\di w,
        \end{align}
        where in the second equality we made the change of variables $w=z+r\theta a$. From \eqref{changov} and the fact that $\rho$ is supported in $\B$ we get
        $$\|\Omega_\theta\|_{\cC^N(U)}\lesssim|r\theta|^{-2k-N}\cdot|r\theta|^{2k}\cdot\|\rho\|_{\cC^N}\cdot\|\Omega\|_{\cC^0}=|r\theta|^{-N}\cdot\|\rho\|_{\cC^N}\cdot\|\Omega\|_{\cC^0}$$
        for every $N\ge1$, where the implicit constant depends only on $k$. The assertion then follows by taking $N=2$.
    \end{proof}
    
    For every $|\theta|\le1$, we have
    \begin{equation} \label{logqtest-claim0}
        |\langle S,\Omega\rangle|\le|\langle S,\Omega_\theta\rangle|+|\langle S,\Omega-\Omega_\theta\rangle|.
    \end{equation}
    We now bound the two terms in the right hand side of the above expression.\\

    \textbf{Claim 2.} We have
    $$|\langle S,\Omega_\theta\rangle|\lesssim\|\Phi\|_{\log^q}|\theta|^{-2}\|S\|_{-2}.$$
    \begin{proof}
        Since $\Omega_\theta$ is smooth, from the definition of $\|\cdot\|_{-2}$ and Claim 1 it follows that
        $$|\langle S,\Omega_\theta\rangle|=\|\Omega_\theta\|_{\mathcal{C}^2}\cdot\left|\left\langle S,\frac{\Omega_\theta}{\|\Omega_\theta\|_{\mathcal{C}^2}}\right\rangle\right| \le \|\Omega_\theta\|_{\mathcal{C}^2}\|S\|_{-2}\lesssim\|\Phi\|_{\log^q}|\theta|^{-2}\|S\|_{-2}.$$
        This proves the desired inequality.
    \end{proof}

    \textbf{Claim 3.} We have
    $$|\langle S,\Omega-\Omega_\theta\rangle|\lesssim \frac{\|\Phi\|_{\log^q}}{(1+|\log{|\theta|}|)^q}.$$
    
    \begin{proof}
        We have to prove the inequality only for currents $S$ such that $\|S\|_*\le1$.

        Since $\Omega$ is $(M',q)$-$\log$-continuous and we are regularizing using translations, we have
        $$\|\Omega-\Omega_\theta\|_{\mathcal{C}^0}\lesssim \frac{M'}{(1+|\log{|\theta|}|)^q} \lesssim \frac{\|\Phi\|_{\log^q}}{(1+|\log{|\theta|}|)^q},$$
        where the implicit constants depend only on $\rho$ and $q$. It follows that we have
        $$|\Omega-\Omega_\theta|\lesssim \frac{\|\Phi\|_{\log^q}}{(1+|\log{|\theta|}|)^q}\omega^s,$$
        where we recall that $\omega$ is the K\"ahler form on $X$. Given a current $S$ with $\|S\|_*\le1$, write $S=S^+-S^-$, with $S^\pm$ positive currents with $\|S^\pm\|\le1$. We then have
        \begin{align*}
            |\langle S,\Omega-\Omega_\theta\rangle|&=|\langle S^+-S^-,\Omega-\Omega_\theta\rangle|\lesssim |\langle S^+,\Omega-\Omega_\theta\rangle|+|\langle S^-,\Omega-\Omega_\theta\rangle|\\
            &\lesssim \frac{\|\Phi\|_{\log^q}}{(1+|\log{|\theta|}|)^q}(|\langle S^+,\omega^s\rangle|+|\langle S^-,\omega^s\rangle|) \lesssim \frac{\|\Phi\|_{\log^q}}{(1+|\log{|\theta|}|)^q}.
        \end{align*}
        All the constants here are independent of $\Phi$. This proves the desired estimate.
    \end{proof}
    
    Combining \eqref{logqtest-claim0} with the estimates in Claims 2 and 3, we obtain
    $$|\langle S,\Omega\rangle| \lesssim \|\Phi\|_{\log^q}|\theta|^{-2}\|S\|_{-2}+\frac{\|\Phi\|_{\log^q}}{(1+|\log{|\theta|}|)^q}.$$
    Choosing $\theta=\|S\|_{-2}^{\frac{1}{4}}$ we get
    \begin{align*}
        \|\Phi\|_{\log^q}|\theta|^{-2}\|S\|_{-2}+\frac{\|\Phi\|_{\log^q}}{(1+|\log{|\theta|}|)^q}&=\|\Phi\|_{\log^q}\|S\|_{-2}^{1/2}+\frac{\|\Phi\|_{\log^q}}{(1+(1/4)|\log{\|S\|_{-2}}|)^q}\\
        &\lesssim \frac{\|\Phi\|_{\log^q}}{(1+|\log{\|S\|_{-2}}|)^q},
    \end{align*}
    where the implicit constant depends only on $q$. This completes the proof.
\end{proof}

\begin{cor} \label{logq_sup_logq}
    Let $\Phi$ be a $\log^q$-continuous $(s,s)$-form on $X$ for some $q>0$. The super-potential of $\ddc\Phi$ is $(M,q)$-$\log$-continuous for some $0<M\lesssim \|\Phi\|_{\log^q}$, where the implicit constant may depend on $q$, but is independent of $\Phi$.

    Moreover, for every current $S\in\cD^0_{k-s}\cap\cD^1_{k-s}$ the value of $\cU_{\ddc\Phi}(S)$ is given by $\langle S,\Phi\rangle$.
\end{cor}

\begin{proof}
    Take a current $S$ in $\cD^0_{k-s}\cap\cD^1_{k-s}$ with normalized potential $U_S$. We have
    $$\cU_{\ddc\Phi}(S)=\langle \ddc\Phi, U_S\rangle=\langle \Phi,\ddc U_S\rangle=\langle \Phi, S\rangle.$$
    The above chain of equalities is valid when $S$ is smooth and we take $U_S$ smooth. The result then follows from Lemma \ref{logq_test} and Remark \ref{dense}. The proof is complete.
\end{proof}

We will also need the two following results about currents with regular super-potentials, see \cite[Proposition 3.4.2]{DS10JAG} and \cite[Theorem 1.1]{DNV18AM}.

\begin{prop} \label{DS10JAG-3_4_2}
    Take $S_1\in\mathcal{D}_{p_1}$ and $S_2\in\mathcal{D}_{p_2}$ with $p_1+p_2\le k$. Suppose $S_1$ has a $\lambda_1$-H\"older-continuous super-potential for some $\lambda_1>0$. Suppose also that $S_2$ has a $\lambda_2$-H\"older-continuous super-potential for some $\lambda_2>0$ or, respectively, a $\log^q$-continuous super-potential for some $q>0$. Then $S_1\wedge S_2$ has a $\lambda$-H\"older-continuous super-potential, where $\lambda=\min\{\lambda_1,1\}\cdot\min\{\lambda_2,1\}/2$, or, respectively, a $\log^q$-continuous super-potential.
\end{prop}

\begin{proof}
    We can assume without loss of generality that $\lambda_1,\lambda_2\le1$.

    Let $\displaystyle\{S_1\}=\sum_{j=1}^{h_{p_1}} a_j\{\alpha_{p_1,j}\}$ and $\displaystyle\{S_2\}=\sum_{l=1}^{h_{p_2}} b_l\{\alpha_{p_2,l}\}$ be the cohomology decompositions of $\{S_1\}$ and $\{S_2\}$. By the linearity of intersection and super-potentials we have
    \begin{equation} \label{swedga}
        \cU_{S_1\wedge S_2}=\cU_{S_1\wedge\left(S_2-\sum_{l=1}^{h_{p_2}}b_l\alpha_{p_2,l}\right)}+\sum_{l=1}^{h_{p_2}}b_l\cdot\cU_{S_1\wedge\alpha_{p_2,l}}.
    \end{equation}

    \textbf{Claim 1.} For every $l=1,\dots,h_{p_2}$ the super-potential $\cU_{S_1\wedge\alpha_{p_2,l}}$ is $\lambda_1$-H\"older-continuous.
    
    \begin{proof}
        Fix $l$ between $1$ and $h_{p_2}$ and set $\alpha:=\alpha_{p_2,l}$. Consider a current $R$ in $\cD^0_{k-p_1-p_2+1}\cap\cD^1_{k-p_1-p_2+1}$ with normalized potential $U_R$. By Remark \ref{dense}, we only have to consider the case where $R$ is smooth. We have
        \begin{equation}\label{intersect_smooth_1}
            \cU_{S_1\wedge\alpha}(R)=\langle S_1\wedge\alpha,U_R\rangle=\langle S_1,U_R\wedge\alpha\rangle=\langle S_1-\sum_{j=1}^{h_{p_1}}a_j\alpha_{p_1,j},U_R\wedge\alpha\rangle+\sum_{j=1}^{h_{p_1}}a_j\langle \alpha_{p_1,j},U_R\wedge\alpha\rangle.
        \end{equation}

        Since $\displaystyle S_1-\sum_{j=1}^{h_{p_1}}a_j\alpha_{p_1,j}$ is exact, we get

        \begin{equation}\label{intersect_smooth_2}
            \langle S_1-\sum_{j=1}^{h_{p_1}}a_j\alpha_{p_1,j},U_R\wedge\alpha\rangle=\langle S_1-\sum_{j=1}^{h_{p_1}}a_j\alpha_{p_1,j},U_{R\wedge\alpha}\rangle=\langle S_1,U_{R\wedge\alpha}\rangle=\cU_{S_1}(R\wedge\alpha),
        \end{equation}
        where $U_{R\wedge\alpha}$ is the normalized potential of $R\wedge\alpha$. Combining \eqref{intersect_smooth_1} and \eqref{intersect_smooth_2}, we obtain

        \begin{equation}\label{intersect_smooth}
            \cU_{S_1\wedge\alpha}(R)=\cU_{S_1}(R\wedge\alpha)+\sum_{j=1}^{h_{p_1}}a_j\langle \alpha_{p_1,j},U_R\wedge\alpha\rangle.
        \end{equation}

        Since $\alpha$ is smooth, it is not difficult to see that $\|R\wedge\alpha\|_{-2}\lesssim \|R\|_{-2}$, so by the regularity of $\cU_{S_1}$ we have $|\cU_{S_1}(R\wedge\alpha)|\lesssim \|R\|_{-2}^{\lambda_1}$. To bound the sum in \eqref{intersect_smooth}, we fix $j$ between $1$ and $h_{p_1}$, set $\beta=\alpha_{p_1,j}$, and give an estimate of
        $$\langle \beta,U_R\wedge\alpha\rangle=\langle \alpha\wedge\beta,U_R\rangle=\cU_{\alpha\wedge\beta}(R).$$
        Since $\alpha\wedge\beta$ is smooth, by Lemma \ref{supofsmooth} we have $|\cU_{\alpha\wedge\beta}(R)| \lesssim \|R\|_{-2}\lesssim\|R\|_{-2}^{\lambda_1}$. This proves the assertion.
    \end{proof}
    By Claim 1, up to adding to the super-potential of $S_1\wedge S_2$ a term which is $\lambda_1$-H\"older-continuous, we can assume without loss of generality that $\{S_2\}=0$. Since to do so we only add to $S_2$ a combination of the $\alpha_{p_2,l}$'s, this will not change $\cU_{S_2}$. If $S_2$ has $\lambda_2$-H\"older continuous super-potential, by symmetry we can assume also that $\{S_1\}=0$. If instead $S_2$ has $\log^q$-continuous super-potential, we need to prove the following.

    \medskip

    \textbf{Claim 2.} If $S_2$ has $\log^q$-continuous super-potential, the super-potential $\cU_{S_2\wedge\alpha_{p_1,j}}$ is $\log^q$-continuous for every $j=1,\dots,h_{p_1}$.
    
    \begin{proof}
        We repeat the same proof of Claim 1. The only difference will come from the contribution of a term of the form $\cU_{S_2}(R\wedge \alpha)$ for some $\alpha$ smooth with bounded $\cC^2$-norm. Since $\cU_{S_2}$ is $\log^q$-continuous, this is still sufficient to prove the assertion.
    \end{proof}
    
    Exchanging $S_1$ and $S_2$ in \eqref{swedga} and using Claim 2, up to adding to the super-potential of $S_1\wedge S_2$ a term which is $\log^q$-continuous, we can assume without loss of generality that $\{S_1\}=0$ even in the case where $S_2$ has $\log^q$-continuous super-potential. Since to do so we only add to $S_1$ a combination of the $\alpha_{p_1,j}$'s, this will not change $\cU_{S_1}$.

    \medskip

    Consider now a current $R$ in $\cD^0_{k-p_1-p_2+1}\cap\cD^1_{k-p_1-p_2+1}$. By Remark \ref{dense}, we only have to consider the case where $R$ is smooth. Moreover, since $\|\cdot\|_{-2}\lesssim\|\cdot\|_*$, by linearity we need to check the statement only for currents $R$ such that $\|R\|_{-2}$ is sufficiently small. Since we are assuming that $\{S_2\}=0$, we have $\cU_{S_1\wedge S_2}(R)=\cU_{S_2}(S_1\wedge R)$.

    \medskip
    
    \textbf{Claim 3.} We have ${\|S_1\wedge R\|_{-4}\lesssim \|R\|_{-2}^{\lambda_1}}$
    
    \begin{proof}
        Take a smooth $(p_2-1,p_2-1)$-form $\Omega$ with $\|\Omega\|_{\cC^4}\le 1$, so $\|\ddc\Omega\|_{\cC^2}\lesssim 1$. We just need to bound $|\langle S_1\wedge R,\Omega\rangle|$. Since we are assuming $\{S_1\}=0$, take a potential $U$ of $S_1$. Take also a normalized potential $U_{R\wedge\ddc \Omega}$ of $R\wedge\ddc\Omega$. We have
        \begin{align*}
            |\langle S_1\wedge R,\Omega\rangle|&=|\langle (\ddc U)\wedge R,\Omega\rangle|=|\langle \ddc (U\wedge R),\Omega\rangle|=|\langle U\wedge R,\ddc\Omega\rangle|\\
            &=|\langle U,R\wedge\ddc\Omega\rangle|=|\langle U,\ddc U_{R\wedge\ddc\Omega}\rangle|=|\langle \ddc U,U_{R\wedge\ddc\Omega}\rangle|\\
            &=|\langle S_1,U_{R\wedge\ddc\Omega}\rangle|=|\cU_{S_1}(R\wedge\ddc\Omega)|\lesssim \|R\wedge\ddc\Omega\|_{-2}^{\lambda_1}\lesssim \|R\|_{-2}^{\lambda_1},
        \end{align*}
        where in the last inequality we used the fact that $\|\ddc\Omega\|_{\cC^2}\lesssim 1$. This proves the desired inequality.
    \end{proof}
    
    Using Claim 3 and Proposition \ref{interpol} we get
    $$\|S_1\wedge R\|_{-2}\lesssim\|S_1\wedge R\|_{-4}^{1/2}\lesssim \|R\|_{-2}^{\lambda_1/2}.$$

    To complete the proof, it suffices to use the regularity of $\cU_{S_2}$.
\end{proof}

\begin{prop} \label{DNV18AM-1_1}
    Let $S$ and $S'$ be positive currents in $\cD_s$ for some $1\le s\le k$. Assume that $S'\le S$. If $S$ has a $\lambda$-H\"older-continuous super-potential for some $0<\lambda\le1$ or, respectively, a $\log^q$-continuous super-potential for some $q>0$, then $S'$ has a $\big(C',\lambda/(50k)\big)$-H\"older-continuous super-potential, or, respectively, a $(C',q)$-$\log$-H\"older-continuous super-potential. The positive constant $C'$ depends on $S$, but is independent of $S'$.
\end{prop}

The case of H\"older-continuity was proved in \cite[Theorem 1.1]{DNV18AM}, up to checking that the constant $C'$ does not depend on $S'$. We will follow their strategy to prove the $\log^q$-continuous case and verify the independence of the constant from $S'$ also for the H\"older-continuous case.

\begin{rmk}
    In \cite{DNV18AM}, the authors define the value of the super-potential of a $(s,s)$-current $S$ at a $(k-s+1,k-s+1)$-current $R$ in the following way: they fix a smooth form $\alpha$ in the same cohomology class of $S$, they consider \emph{any} (not necessarily normalized) potential $U_R$ of $R$, and they define $\cU_S(R):=\langle S-\alpha, U_R\rangle$. If $\displaystyle\{S\}=\sum_{j=1}^{h_s} a_j\{\alpha_{s,j}\}$ and we choose $\displaystyle\alpha:=\sum_{j=1}^{h_s} a_j\alpha_{s,j}$, then the definition of $\cU_S$ given in \cite{DNV18AM} will coincide with our definition given using the normalized potential.
\end{rmk}

\begin{proof}[Proof of Proposition \ref{DNV18AM-1_1}.]
    Let $\Pi:\widehat{X\times X}\longrightarrow X\times X$ be the blow-up of $X\times X$ along the diagonal $\Delta$. The compact complex manifold $\widehat{X\times X}$ is a K\"ahler manifold \cite{B56ASENS}, and we consider on it the metric induced by a fixed K\"ahler form $\widehat{\Omega}$. Set $\widehat{\Delta}:=\Pi^{-1}(\Delta)$ and denote by $\widehat{\Delta}_\ep$ the $\ep$-neighbourhood of $\widehat{\Delta}$. Denote by $\pi_j$ the projections from $X\times X$ onto its factors for $j=1,2$ and deﬁne $\Pi_j:=\pi_j\circ\Pi$.

    Choose a real smooth closed $(1,1)$-form $\widehat{\beta}$ on $\widehat{X\times X}$ which is cohomologous to $[\widehat{\Delta}]$. We can write $\widehat{\beta}-[\widehat{\Delta}]=\ddc\widehat{u}$, where $\widehat{u}$ is a quasi-p.s.h.\ function on $\widehat{X\times X}$. This implies that $\widehat{u}$ is smooth outside $\widehat{\Delta}$ and $\widehat{u}-\log{\dist(\cdot,\widehat{\Delta})}$ is a bounded function near $\widehat{\Delta}$. Subtracting from $\widehat{u}$ a constant allows us to assume that it is negative.
    
    The following lemma gives an estimate of the mass of $\Pi_1^*(S)$ near $\widehat{\Delta}$ when $\cU_S$ is regular.
    
    \begin{lm}
        For every current $\tilde{S}$ in $\cD_s$, define
        \begin{equation}\label{thetaessedef}
            \theta_{\tilde{S}}(\ep):=\sup_{\tilde{R}} \int_{\widehat{\Delta}_\ep}-\widehat{u}\,\widehat{\Omega}^{k-1}\wedge\Pi_1^*(\tilde{S})\wedge\Pi_2^*(\tilde{R}),
        \end{equation}
        where the supremum is taken over all smooth positive closed $(k-s+1,k-s+1)$-forms $\tilde{R}$ on $X$ with $\|\tilde{R}\|\le1$. We have:
        \begin{equation} \label{thetaesse}
            \begin{cases}
                \theta_S(\ep)\lesssim \ep^{\lambda/2}&\text{if }\cU_S\text{ is }\lambda\text{-H\"older-continuous};\\
                \theta_S(\ep)\lesssim \frac{1}{(1+|\log{\ep}|)^q}&\text{if }\cU_S\text{ is }\log^q\text{-continuous}.
            \end{cases}
    \end{equation}
    The implicit constants may depend on $S$.
    \end{lm}
    
    \begin{proof}
        Observe that $\theta_S$ is monotone in $\ep$, so we just need to give an estimate of $\theta_S(\ep_n)$, where $\ep_n=e^{-2n-3n_0}$ and $n_0$ is a fixed positive integer that depends only on $X$. From the proofs of \cite[Proposition 2.7 and Proposition 2.8]{DNV18AM} it follows that, up to adding a term which decreases as $e^{-n}$, we just have to give an estimate of $|\tilde{\cU}(\tilde{R}_{n,m})|$. Here, the $\tilde{R}_{n,m}$'s are currents on $\widehat{X\times X}$ with $\|\tilde{R}_{n,m}\|_{-2}\lesssim e^{-n}$ and $\|\tilde{R}_{n,m}\|_*\lesssim 1$, and we have $\tilde{\cU}=\cU_S\circ(\Pi_1)_*$. From Lemma \ref{star_lip} it follows that if $\cU_S$ is $\lambda$-H\"older-continuous (respectively, $\log^q$-continuous), then $\tilde{\cU}$ is $\lambda$-H\"older-continuous (respectively, $\log^q$-continuous). Therefore, we obtain $|\tilde{\cU}(\tilde{R}_{n,m})|\lesssim e^{-\lambda n}$ (respectively, $|\tilde{\cU}(\tilde{R}_{n,m})|\lesssim \frac{1}{(1+n)^q}$). As a consequence, we have $\theta_S(\ep_n)\lesssim e^{-\lambda n}$ (respectively, $\theta_S(\ep_n)\lesssim \frac{1}{(1+n)^q}$). By the monotonicity of $\theta_S$, we get \eqref{thetaesse}.
    \end{proof}

    \emph{End of the proof of Proposition \ref{DNV18AM-1_1}.} From the fact that $\widehat{\Omega}$ is K\"ahler, $\widehat{u}$ is negative, the currents $\tilde{R}$ in \eqref{thetaessedef} are positive, and $S'\le S$, it follows that we have $\theta_{S'}\le\theta_S$. So from \eqref{thetaesse} we get
    \begin{equation} \label{thetaesseprime}
        \theta_{S'}(\ep)\lesssim \ep^{\lambda/2}\qquad\text{or}\qquad\theta_{S'}(\ep)\lesssim \frac{1}{(1+|\log{\ep}|)^q},
    \end{equation}
    depending on the regularity of $\cU_S$. Observe that the implicit constants in \eqref{thetaesseprime} do not depend on $S'$.

    Consider a smooth form $R$ in $\cD_{k-s+1}^0\cap\cD_{k-s+1}^1$ and set $\ep_0:=\|R\|_{-2}^{1/(25k)}$. As in the proof of \cite[Proposition 2.8]{DNV18AM}, we write
    \begin{equation}\label{dom_deco}
        \cU_{S'}(R)=I_1+I_2,
    \end{equation}
    where $I_1$ and $I_2$ are two quantities for which we have
    \begin{equation}\label{dommm}
        |I_1| \lesssim \theta_{S'}(4\ep_0)+\ep_0\qquad\text{and}\qquad|I_2|\lesssim \ep_0
    \end{equation}
    if $\ep_0$ is small enough (independently of $S'$). The implicit constants in \eqref{dommm} depend on $S'$. From \eqref{thetaesseprime}, \eqref{dom_deco} and \eqref{dommm} it follows that we have
    \begin{equation} \label{dom_final}
        |\cU_{S'}(R)|\lesssim \|R\|_{-2}^{\lambda/(50k)}\qquad\text{or}\qquad|\cU_{S'}(R)|\lesssim \frac{1}{(1+|\log{\|R\|_{-2}}|)^q},
    \end{equation}
    depending on the regularity of $\cU_S$. Observe that, a priori, the implicit constants in \eqref{dom_final} depend on $S'$. More precisely, if we have $\{S'\}=\displaystyle \sum_{j=1}^{h_s} a_j\{\alpha_{s,j}\}$ and we set $\displaystyle\alpha:=\sum_{j=1}^{h_s} a_j\alpha_{s,j}$, then the constants in \eqref{dommm}, and therefore also the constants in \eqref{dom_final}, depend on $\|S'-\alpha\|_*$ and $\|\alpha\|_{\cC^0}$. We need to bound these quantities independently of $S'$.
    
    We have $\alpha\le c\|\alpha\|_{\cC^0}\omega^s$, where the constant $c$ depends only on $X$. So we can write $S'-\alpha$ as the difference of two positive closed currents in the following way:
    $$S'-\alpha=(S+c\|\alpha\|_{\cC^0}\omega^s-\alpha)-(S-S'+c\|\alpha\|_{\cC^0}\omega^s).$$
    Therefore, we have
    \begin{align}\label{viet-anh_hint}
        \|S-\alpha\|_*&\le \langle S+c\|\alpha\|_{\cC^0}\omega^s-\alpha,\omega^{k-s}\rangle+\langle S-S'+c\|\alpha\|_{\cC^0}\omega^s,\omega^{k-s}\rangle\\ \nonumber
        &\le \langle 2S, \omega^{k-s}\rangle+\langle 2c\|\alpha\|_{\cC^0}\omega^s-\alpha,\omega^{k-s}\rangle,
    \end{align}
    where we used the positivity of $S'$. The term $\langle 2S, \omega^{k-s}\rangle$ does not depend on $S'$. The term $\langle 2c\|\alpha\|_{\cC^0}\omega^s-\alpha,\omega^{k-s}\rangle$ is bounded by $\|\alpha\|_{\cC^0}$, so we just have to bound the $\cC^0$-norm of $\alpha$ independently of $S'$. To do so, it is sufficient to bound each $a_j$. Fix $j_0$ between $1$ and $h_s$. By Poincaré duality, there exists a smooth closed $(k-s,k-s)$-form $\beta$, independent of $S'$, such that $\langle\alpha_j,\beta\rangle$ is equal to $1$ if $j=j_0$, and $0$ otherwise. Hence, we have $a_{j_0}=\langle S',\beta\rangle$ and
    $$|a_{j_0}|=|\langle S',\beta\rangle| \lesssim \|S'\|\le \|S\|,$$
    where we used the fact that $S'\le S$. Since the implicit constant in the chain of inequalities above depends only on $\beta$, which is independent of $S'$, the proof is complete.
\end{proof}

\subsection{The operators $\cL_\theta$}\label{linear_operators}

We now introduce a family of operators on currents that we will need to prove the main result of Section \ref{st_ests}. For $0\le s\le k$, we consider the family of linear operators $\mathcal{L}_\theta:\mathcal{D}_s \longrightarrow \mathcal{D}_s$, with $\theta$ in $\mathbb{P}^1=\mathbb{C}\cup\{\infty\}$, defined as in \cite[Section 2.4]{DS10JAG}. We will not need the precise definition, which is quite technical, but just some of their properties that we now recall, see \cite[Sections 2.3 and 2.4]{DS10JAG}.

\begin{prop}
    We have that:
    \begin{nlist}
        \item $\mathcal{L}_\theta$ depends only on $|\theta|$;
        \item $\mathcal{L}_\theta$ is continuous with respect to the $*$-topology for every $\theta\in\mathbb P^1$;
        \item $\mathcal{L}_0(S)=S$;
        \item $\mathcal{L}_\theta=\mathcal{L}_\infty$ for $\theta$ outside the unit disk;
        \item $\mathcal{L}_\theta(S)$ is in the same cohomology class of $S$ for every $\theta\in\mathbb P^1$ and $S\in\mathcal{D}_s$.
    \end{nlist}
\end{prop}

For $S\in\mathcal{D}_s$, we set $S_\theta:=\mathcal{L}_\theta(S)$. We have $\{S_\theta\}=\{S\}$. When $\{S\}=0$, we call $(S_\theta)_{\theta\in\mathbb{P}^1}$ the \textit{special structural line} associated with $S$.

\begin{lm} \label{DS10JAG-2_3_2}
    For $1\le m\le k$, put $q_m=\frac{k+1}{k-m+1}$. Take $S\in\mathcal{D}_s^0$. There exists a constant $c>0$, independent of $S$, such that:
    \begin{nlist}
        \item $\|\mathcal{L}_\infty(S)\|_{L^{q_1}}\le c\|S\|_*$;
        \item $\|\mathcal{L}_\infty(S)\|_{L^{q_{m+1}}}\le c\|S\|_{L^{q_m}}$ for every $1\le m<k$;
        \item $\|\mathcal{L}_\infty(S)\|_{L^\infty}\le c\|S\|_{L^q}$ for every $k+1\le q\le+\infty$.
    \end{nlist}
\end{lm}

\begin{lm} \label{DS10JAG-2_4_3}
    Take $S\in\mathcal{D}_s$. The current $S_\theta$ depends continuously on $(\theta,S)$ with respect to the standard topology on $\mathbb P^1$ and the $*$-topology on $\mathcal{D}_s$. In particular, we have $\|S_\theta\|_*\le c\|S\|_*$ for some constant $c>0$ independent of $S$ and $\theta$. Moreover, we have $\dist_2(S_\theta,S) \le c|\theta|\cdot\|S\|_*$ for some $c>0$ independent of $S$ and $\theta$.
\end{lm}

We also need the following results which in \cite[Section 3.2]{DS10JAG} are stated for smooth currents, and can be extended to all currents with continuous super-potential, see \cite[Proposition 2.1]{DS10CM}.

\begin{lm} \label{DS10JAG-3_2_4}
    Let $(R_\theta)_{\theta\in\mathbb{P}^1}$ be the special structural line associated to a current $R\in\mathcal{D}^0_{k-s+1}$ with continuous super-potential. Let $S$ be a current in $\mathcal{D}_s$. Then $\theta\mapsto\mathcal{U}_S(R_\theta)$ is a continuous d.s.h.\ function on $\mathbb{P}^1$ which is constant on $\{|\theta|\ge1\}$ and depends only on $|\theta|$. Moreover, we have
    $$\|\mathrm{dd^c_\theta}\,\mathcal{U}_S(R_\theta)\|_*\le c\|S\|_*\|R\|_*,$$
    where $c>0$ is a constant independent of $R$ and $S$.
\end{lm}

\begin{lm} \label{DS10JAG-3_2_5}
    Take $S\in\mathcal{D}_s^0$ and let $\mathcal{U}_{S_\theta}$ be the normalized super-potential of $S_\theta$. If ${R\in\mathcal{D}_{k-s+1}}$ has continuous super-potential, then we have $\mathcal{U}_{S_\theta}(R)=\mathcal{U}_S(R_\theta)$ for every $\theta\in\mathbb{P}^1$.
\end{lm}
\vspace*{-\baselineskip}

\subsection{Skoda-type estimates} \label{st_ests}

The goal of this section is to prove the following ``Skoda-type'' estimate, which is the analogue of \cite[Theorem 3.2.6]{DS10JAG} and \cite[Proposition 2.1]{DS10CM} for currents with $\log$-continuous super-potentials.

\begin{prop} \label{skoda_clicks}
    There is a constant $A\ge1$, depending only on $X$, such that, if $R$ is a current in $\mathcal{D}^0_{k-s+1}\cap\mathcal{D}^1_{k-s+1}$ whose super-potential $\mathcal{U}_R$ is $(M,q)$-$\log$-continuous for some $M,q>0$, then for every current $S$ in $\mathcal{D}^0_s\cap\mathcal{D}^1_s$ the super-potential $\mathcal{U}_S$ of $S$ satisfies
    $$|\mathcal{U}_S(R)| \le A\big(1+M^{\frac{1}{q+1}}\big).$$
\end{prop}
\vspace*{-\baselineskip}

In the proof of Proposition \ref{skoda_clicks} we will need the following lemma, which is a version of the classic Skoda estimate \cite{S72BSMF}, \cite[Theorem 4.4.5]{H90Book}, see for instance \cite[Lemma 2.2.4]{DS10JAG}.

\begin{lm} \label{DS10JAG-2_2_4}
    Let $u$ be a d.s.h.\ function on $\mathbb{P}^1=\mathbb{C}\cup\{\infty\}$. Assume that $u$ vanishes outside the unit disk of $\mathbb{C}$ and that $\ddc u$ is a measure with $\|\ddc u\|_*\le1$. Then there are constants $\lambda>0$ and $\beta>0$ independent of $u$ such that
    $$\|e^{\lambda|u|}\|_{L^1(\fs)} \le \beta.$$
\end{lm}
\vspace*{-\baselineskip}

\begin{proof}[Proof of Proposition \ref{skoda_clicks}.]
    Multiplying $S$ by a sufficiently small constant allows us to assume that $\|S\|_*\le c^{-k-3}$, where $c$ is as in Lemma \ref{DS10JAG-2_3_2}. We define
    $$S_0:=S\quad\text{and}\quad S_{i+1}:=\mathcal{L}_\infty(S_i)\quad\text{for}\quad0\le i\le k+1.$$
    We define also
    $$u_i(\theta):=\mathcal{U}_{\mathcal{L}_\theta(S_i)}(R)\quad\text{and}\quad{m_i:=u_i(0)=u_{i-1}(\infty)}.$$
    Lemmas \ref{DS10JAG-2_3_2}, \ref{DS10JAG-2_4_3}, \ref{DS10JAG-3_2_4}, and \ref{DS10JAG-3_2_5} give us
    \begin{equation} \label{justthemiddle}
        \|S_i\|_*\le1/c,\quad\|\ddc u_i\|_*\le1\quad\text{for every }0 \le i\le k+2,\quad\text{and}\quad\|S_{k+2}\|_{L^\infty} \le 1.
    \end{equation}
    The last inequality implies that $|m_{k+2}|$ is bounded by a constant independent of $S$ and $R$. Indeed, $R$ always admits a potential $U_R$ of bounded $L^1$-norm, and we have $\{S_{k+2}\}=0$, so we get $m_{k+2}=\langle S_{k+2}, U_R\rangle$.

    We need to show that $|m_0| \le A\big(1+M^{\frac{1}{q+1}}\big)$ for some constant $A\ge1$. We can assume, without loss of generality, that we have $M\ge1$. It is enough to check that we have
    $$|m_i-m_{i+1}| \le A\big(1+M^{\frac{1}{q+1}}\big)$$
    for some constant $A\ge1$ depending only on $X$. We have $m_i-m_{i+1}=v_i(0)$ where ${v_i:=u_i-m_{i+1}}$. Lemmas \ref{DS10JAG-3_2_4} and \ref{DS10JAG-3_2_5} and the second inequality in \eqref{justthemiddle} imply that the $v_i$'s are continuous, vanish outside the unit disc and satisfy $\|\ddc v_i\|_*\le1$. By Lemma \ref{DS10JAG-2_2_4} we have $\|e^{\lambda|v_i|}\|_{L^1(\fs)}\le \beta$ for some universal constants $\beta>0$ and $\lambda>0$. Arguing by contradiction, we deduce that there is a $\theta$ satisfying
    $$|\theta|\le e^{1-M^{\frac{1}{q+1}}}\quad\text{and}\quad|v_i(\theta)|\le A\big(1+M^{\frac{1}{q+1}}\big)$$
    for a constant $A\ge1$ sufficiently large (it is enough to choose $A\ge\max\{2,C,\lambda\}/\lambda$, with $C$ such that ${\pi\cdot e^{C+2}>10\beta}$). Finally, using Lemma \ref{shuffle} and the $\log$-continuity of $\mathcal{U}_R$, we get
    \begin{align*}
        |v_i(0)-v_i(\theta)|&=|\mathcal{U}_{S_i}(R)-\mathcal{U}_{\mathcal{L}_\theta(S_i)}(R)|=\left|\mathcal{U}_R(S_i)-\mathcal{U}_R\big(\mathcal{L}_\theta(S_i)\big)\right|\\
        &\le w_{M,q}\Big(\dist_2\big(S_i,\mathcal{L}_\theta(S_i)\big)\Big) \le w_{M,q}(|\theta|) \le M^{\frac{1}{q+1}}.
    \end{align*}

    Therefore, we obtain $|v_i(0)|\le 2A\big(1+M^{\frac{1}{q+1}}\big)$. Up to replacing $A$ with $(2k+4)A+|m_{k+2}|$, this completes the proof.
\end{proof}

\section{Automorphisms of compact K\"ahler manifolds}\label{kahler}

\subsection{Equidistribution results} \label{kahler_sub1}
The goal of this section is to prove Theorem \ref{goal_kahler}.

\smallskip

By Poincaré duality, the dynamical degree $d_s$ of $f$ is equal to the dynamical degree $d_{k-s}(f^{-1})$ of $f^{-1}$. Fix $\ep>0$. Set $\delta_s:=d_s+\ep$ for every $0\le s\le k$ with $s\not=p$, and $\delta_p=d_p$. Since the mass of a positive closed current can be computed cohomologically, for every $0\le s\le k$ and for every $S\in\mathcal{D}_s$ we have
\begin{equation}\label{dynamic-of-f}
    \|(f^n)^*(S)\|_* \le c\delta_s^n \|S\|_*
\end{equation}
for some constant $c>0$ independent of $S$ and $n$. We can assume $c\ge1$ without loss of generality.

\medskip

We will need the following lemma, see for instance \cite[Lemma 2.8 (i), Lemma 3.1, and Corollary 3.3]{BD26AM}.

\begin{lm} \label{BD23-3_3}
    Let $f,d_p,\delta,T_+$ be as in the introduction, and let $S$ be a current in $\mathcal{D}_p$.
    \begin{nlist}
        \item $T_+$ is the unique positive closed current in $\{T_+\}$, and it has a H\"older-continuous normalized super-potential.
        \item We have $d_p^{-n}(f^n)^*S\rightharpoonup rT_+$. Here, the constant $r=r(S)$ depends linearly on $\{S\}$. More precisely, $r$ is the constant such that $d_p^{-n}(f^n)^*\{S\}\rightarrow r\{T_+\}$.
    \end{nlist}
    Let now $S$ be such that $\|S\|_*\le1$ and $r(S)=0$. Let $\xi$ be a smooth $(k-p,k-p)$-form. Then
    $$|\langle d_p^{-n}(f^n)^*(S),\xi\rangle| \le c\|\xi\|_{\mathcal{C}^2}(\delta/d_p)^n,$$
    where $c>0$ is a constant independent of $S,\xi$, and $n$.
\end{lm}

\begin{prop} \label{pushforward_kahler}
    Let $R$ be a current in $\mathcal{D}^0_{k-s+1}$, $0\le s\le k$, whose super-potential $\mathcal{U}_R$ is $(M,q)$-$\log$-continuous for some $M,q>0$. Then there exists a constant $C_q\ge1$ such that the super-potential $\mathcal{U}_{f_*^n(R)}$ of $f_*^n(R)$ is $(C_q n^q \delta_s^n M,q)$-$\log$-continuous for every $n\ge1$.
\end{prop}

\begin{proof}
    Assume without loss of generality that $M=1$. Take $S\in\cD^0_s\cap\cD^1_s$. From Lemma \ref{pb-commute} we have
    $$\mathcal{U}_{f^n_*(R)}(S)=\mathcal{U}_R\big((f^n)^*(S)\big)=c\delta_s^n\cdot\mathcal{U}_R\left(\frac{(f^n)^*}{c\delta_s^n}(S)\right).$$
    From \eqref{dynamic-of-f} we have that $\displaystyle \left\|\frac{(f^n)^*}{c\delta_s^n}(S)\right\|_*\le1$ for some $c\ge1$, so we can use the $\log^q$-continuity of $\cU_R$ to obtain
    \begin{equation}\label{beforeclaim}
        |\mathcal{U}_{f^n_*(R)}(S)|=c\delta_s^n\left|\mathcal{U}_R\left(\frac{(f^n)^*}{c\delta_s^n}(S)\right)\right| \le \frac{c\delta_s^n}{\left(1+\left|\log\left\|\frac{(f^n)^*}{c\delta_s^n}(S)\right\|_{-2}\right|\right)^q}.
    \end{equation}

    The desired result is then a consequence of the following inequality:

    \begin{equation} \label{ennequ}
        1+|\log\|S\|_{-2}| \lesssim n\left(1+\left|\log\left\|\frac{(f^n)^*}{c\delta_s^n}(S)\right\|_{-2}\right|\right).
    \end{equation}
    The rest of the proof is devoted to proving \eqref{ennequ}.

    Applying Lemma \ref{star_lip} to $\tau=f^{-1}$, we have that the operator $f^*/\delta_s$ is Lipschitz, i.e., $(C,1)$-H\"older-continuous with respect to $\|\cdot\|_{-2}$ in $\cD_s$ for some constant $C>0$, and without loss of generality we can assume $C\ge1$. By induction, we get
    \begin{equation} \label{liporhol}
        \left\|\frac{(f^n)^*}{c\delta_s^n}(S)\right\|_{-2}\le \frac{C^n}{c}\|S\|_{-2}\le C^n\|S\|_{-2}.
    \end{equation}
    We also have $\|\cdot\|_{-2}\le \tilde{c}\|\cdot\|_*$, and we can assume without loss of generality that $\tilde{c}\ge1$.
    
    Fix a constant $\tilde{C}>1+\max\{\log{\tilde{c}},2\log{C}\}$. If we have $1+|\log\|S\|_{-2}|\le\tilde{C}n$, then \eqref{ennequ} trivially holds. Suppose that we have $1+|\log\|S\|_{-2}|>\tilde{C}n$. It follows that we have
    $$\log{\|S\|_{-2}}<1-\tilde{C}n\qquad\text{or}\qquad\log{\|S\|_{-2}}>\tilde{C}n-1>\log{\tilde{c}}.$$
    The second of the two inequalities above cannot hold since $\|S\|_{-2}\le \tilde{c}\|S\|_*\le \tilde{c}$. We deduce that we have $\log{\|S\|_{-2}}<1-\tilde{C}n<-2n\log{C}$, from which we get $\|S\|_{-2}\le C^{-2n}\le C^{-n}$. Therefore, from \eqref{liporhol} we obtain $\displaystyle \left\|\frac{(f^n)^*}{c\delta_s^n}(S)\right\|_{-2}\le1$. Then \eqref{ennequ} is equivalent to
    \begin{equation} \label{ennequprimo}
        1-\log\|S\|_{-2} \lesssim n\left(1-\log\left\|\frac{(f^n)^*}{c\delta_s^n}(S)\right\|_{-2}\right).
    \end{equation}
    To prove \eqref{ennequprimo}, we observe that from \eqref{liporhol} we get
    \begin{equation} \label{ennequsecondo}
        \log{\left\|\frac{(f^n)^*}{c\delta_s^n}(S)\right\|_{-2}} \le n\log{C}+\log{\|S\|_{-2}}\le\frac{1}{2}\log{\|S\|_{-2}},
    \end{equation}
    where in the second inequality we used the fact that we have $\log{\|S\|_{-2}}<-2n\log{C}$. Since \eqref{ennequsecondo} implies \eqref{ennequprimo}, the proof is complete.
\end{proof}

As a consequence of Proposition \ref{pushforward_kahler}, we have the following result of convergence for currents with $\log$-continuous super-potentials. The case of H\"older-continuous super-potentials is contained in the proof of \cite[Proposition 3.1]{DS10CM}.

\begin{thm} \label{ennequsuqupiuuno}
    For every $M,q>0$, for every $R\in\mathcal{D}_{k-p+1}^0\cap\mathcal{D}_{k-p+1}^1$ with $(M,q)$-$\log$-continuous super-potential $\mathcal{U}_R$, and for every $S\in\mathcal{D}_p^0\cap\mathcal{D}_p^1$, we have
    $$\left|\mathcal{U}_R\left(\frac{(f^n)^*}{d_p^n}(S)\right)\right| \lesssim \left(1+M^{\frac{1}{q+1}}\right)\left(\dfrac{\delta}{d_p}\right)^{\frac{nq}{q+1}},$$
    where the implicit constant depends only on $f$ and $q$.
\end{thm}

\begin{proof}
    From Lemmas \ref{pb-commute} and \ref{shuffle} we have
    $$\mathcal{U}_R\left(\frac{(f^n)^*(S)}{d_p^n}\right)=\frac{1}{d_p^n}\mathcal{U}_{f^n_*(R)}(S)=\mathcal{U}_S\left(\frac{f^n_*(R)}{d_p^n}\right).$$

    By linearity, we have that $\mathcal{U}_S\left(\frac{f^n_*}{d_p^n}(R)\right)=c\left(\frac{\delta'}{d_p}\right)^n\mathcal{U}_S\left(\frac{f^n_*}{c(\delta')^n}(R)\right)$ (recall that $\delta'<\delta<d_p$). We also have, again by linearity, that $\mathcal{U}_{\frac{f^n_*}{c(\delta')^n}(R)}=\frac{1}{c(\delta')^n}\cdot\mathcal{U}_{f^n_*(R)}$. Proposition \ref{pushforward_kahler} then implies that $\mathcal{U}_{\frac{f^n_*}{c(\delta')^n}(R)}$ is $\left(C_q n^q\left(\frac{d_p}{\delta'}\right)^n M,q\right)$-$\log$-continuous. Thus, we can apply Proposition \ref{skoda_clicks} to the current $\dfrac{f^n_*}{c(\delta')^n}(R)$, obtaining
    $$\left|\mathcal{U}_S\left(\frac{f^n_*}{d_p^n}(R)\right)\right|=\left|c\left(\frac{\delta'}{d_p}\right)^n\mathcal{U}_S\left(\frac{f^n_*}{c(\delta')^n}(R)\right)\right| \le c\left(\frac{\delta'}{d_p}\right)^nA+cA(C_q M)^{\frac{1}{q+1}} n^{\frac{q}{q+1}}\left(\frac{\delta'}{d_p}\right)^{\frac{nq}{q+1}},$$
    where $A$ is given by Proposition \ref{skoda_clicks}. From $\delta'<\delta$ we have $n\lesssim (\delta/\delta')^n$. This concludes the proof.
\end{proof}

We will also need the following result, which is a consequence of the proof of \cite[Proposition 4.2.2]{DS10JAG}. Observe that in the statement the pull-back is acting on the current with regular super-potential. So, Proposition \ref{DS10JAG_4-2-2} alone would not be sufficient to prove Theorem \ref{goal_kahler}, where the regularity is on the test form, and Theorem \ref{ennequsuqupiuuno} is also required.

\begin{prop} \label{DS10JAG_4-2-2}
    Let $S$ be a current in $\mathcal{D}_p$ with continuous super-potential, and let $r$ be the constant such that $d_p^{-n}(f^n)^*(S)$ converge to $rT_+$. Let $R$ be a current in $\mathcal{D}^0_{k-p+1}\cap\mathcal{D}^1_{k-p+1}$. Let $\mathcal{U}_+,\mathcal{U}_n$ be the super-potentials of $T_+$ and $d_p^{-n}(f^n)^*(S)$, respectively. Then
    $$|\mathcal{U}_n(R)-\mathcal{U}_+(R)|\lesssim (\delta/d_p)^n,$$
    where the implicit constant depends only on $f$ and the cohomology class of $S$.
\end{prop}

We are now ready to prove Theorem \ref{goal_kahler}.

\begin{proof}[Proof of Theorem \ref{goal_kahler}.]
    We can rescale and assume without loss of generality that $\|\Phi\|_q\le1$, from which we have $\|\ddc\Phi\|_*\le1$, $M\le1$, and $\|\Phi\|_\infty\le1$, where $M>0$ is a constant such that $\Phi$ is $(M,q)$-$\log$-continuous. Fix a smooth closed form $\Psi$ such that $\Phi-\Psi$ is a normalized potential of $\ddc\Phi$. If we normalize this potential using Poincaré duality, we have that $\|\Psi\|_{\mathcal{C}^2} \lesssim \|\Phi\|_\infty \le 1$, where the implicit constant depends only on $X$ and the choice we made for the normalization. Since we have $M\le1$ and $\Psi$ has bounded $\mathcal{C}^2$ norm, it follows that $\Phi-\Psi$ is a $(\tilde{M},q)$-$\log$-continuous form for some $0<\tilde{M}\lesssim1$.
    
    Up to rescaling, we can apply Lemma \ref{BD23-3_3} to $S-rT_+$ and  we deduce that
    $$|\langle d_p^{-n}(f^n)^*(S)-rT_+,\Psi\rangle|\lesssim (\delta/d_p)^n,$$
    where the implicit constant depends only on $f$ and $r$. Therefore, up to replacing $\Phi$ with $\Phi-\Psi$, we can assume without loss of generality that $\Phi$ is the normalized potential $U_R$ of $R=\ddc\Phi$.

    \smallskip

    Now we prove inequality \eqref{conv-kahler} in the case when $S$ is smooth. Let $\mathcal{U}_n$ and $\mathcal{U}_+$ be the normalized super-potentials of $d_p^{-n}(f^n)^*(S)$ and $T_+$ respectively. We have that the left hand side of \eqref{conv-kahler} is equal to $|\mathcal{U}_n(R)-r\mathcal{U}_+(R)|$. Observe that this quantity is well defined since $d_p^{-n}(f^n)^*(S)$ is smooth for every $n$ and $T_+$ has a continuous super-potential. Since $S$ is smooth, by Lemma \ref{supofsmooth} the super-potential of $S$ is continuous. We then conclude this case appling Proposition \ref{DS10JAG_4-2-2}.
    
    Now we do the general case. Since we proved the estimate for smooth forms, we can subtract from $S$ a smooth closed form and assume without loss of generality $\{S\}=0$, which implies $r=0$. By Corollary \ref{logq_sup_logq} we have that $\mathcal{U}_R$ is $(M',q)$-$\log$-continuous for some constant $M'\lesssim\|\Phi\|_{\log^q}\le\|\Phi\|_q\le1$. Again By Corollary \ref{logq_sup_logq}, we also have that the left hand side of \eqref{conv-kahler} is equal to
    $$\left|\left\langle\frac{(f^n)^*}{d_p^n}(S),\Phi\right\rangle\right|=\left|\mathcal{U}_R\left(\frac{(f^n)^*}{d_p^n}(S)\right)\right|.$$
    We then conclude the proof applying Theorem \ref{ennequsuqupiuuno}.
\end{proof}

\subsection{An application: exponential mixing of all orders} \label{kahler_sub2}

Exponential mixing of all orders for functions bounded in norm $\|\cdot\|_q$ has been proved for the case of equilibrium states of endomorphisms of $\Pb^k$ by Bianchi-Dinh, see \cite[Theorem 5.3]{BD24GAFA}, although for the equilibrium measure the rate is not optimal (compare \cite[Theorem 3.4]{DNS10JDG}). We can apply the results of this section to obtain exponential mixing of all orders for this class of observables also in the case of automorphisms of K\"ahler manifolds. As a consequence, one can also obtain the central limit theorem for the same class of observables, see \cite{BG20JAM}.

\begin{thm} \label{expmix}
Let $f, d_p, \delta,T_\pm$ be as in the introduction, and define the measure $\mu:=T_+\wedge T_-$. Then, $\mu$ is exponentially mixing of all orders for all observables bounded in norm $\|\cdot\|_q$. More precisely, for every integers $\kappa\in \N^*$,  $0=n_0\leq n_1\leq \ldots\leq n_\kappa$ and every functions $\varphi_0,\varphi_1,\dots, \varphi_\kappa$ with $\|\varphi_j\|_q<+\infty$ for every $j$, we have
$$\left| \int \left(\prod_{j=0}^\kappa(\varphi_j\circ f^{n_j})\right) \,\di\mu -\prod_{j=0}^\kappa  \int \varphi_j \,\di \mu \right| \leq C_{\delta,\kappa,q}  \left(\frac{\delta}{d_p}\right)^{\frac{q}{q+1}\min_{0\leq j\leq \kappa-1}(n_{j+1}-n_j)/2}\prod_{j=0}^\kappa \|\varphi_j\|_q,$$
where  $C_{\delta,\kappa,q}>0$ is a constant independent of $n_1,\dots,n_\kappa,\varphi_0,\dots,\varphi_\kappa$.
\end{thm}

\begin{rmk}
    Exponential mixing of all orders in the case of automorphisms of compact K\"ahler manifolds has already been proved by the author and Wu for all d.s.h.\ functions in \cite{VW25Ax}, see also \cite{W21AM} for the case of two observables. However, we remark that for that class of observables the rate was not optimal. In particular, if we set $\rho$ to be the maximal modulus of the eigenvalues of $f^*$ acting on $H^{p,p}(X,\R)$ which are smaller than $d_p$, then in \cite{VW25Ax} it was not possible to get arbitrarily close to the rate $\theta:=(\max\{d_{p-1},d_{p+1},\rho\}/d_p)^{1/2}$. On the other hand, choosing $\delta$ sufficiently small and $q$ sufficiently large, in Theorem \ref{expmix} we are able to get a rate close to $\theta$, see also \cite[Theorem 1.2]{BD26AM} for the case of $\cC^2$ observables. It is still an open question whether $\theta$ is the optimal rate, or if it could be improved, for example to $\theta^2$.
\end{rmk}

\begin{proof}[Sketch of the proof of Theorem \ref{expmix}.]
    Consider the K\"ahler manifold $X\times X$ and its automorphism $F$ defined by $$F(z,w):=\big(f(z),f^{-1}(w)\big).$$ Using K\"unneth formula, one can show that the dynamical degree of order $k$ of $F$ is equal to $d_p^2$ (see also \cite[Section 4]{DS10CM}), which is an eigenvalue of multiplicity $1$ of $F^*$, and that all the others dynamical degrees and the eigenvalues of $F^*$ on $H^{k,k}(X\times X,\mathbb{R})$, except for $d_p^2$, are strictly smaller than $d_p\delta_0$. Hence $F$, $d_p\delta'$ and $d_p\delta$ satisfy the same conditions of $f$, $\delta'$ and $\delta$.

    It is not hard to see that the Green $(k,k)$-currents of $F$ and $F^{-1}$ are $\mathbb T_+:=T_+\otimes T_-$ and $\mathbb T_-:= T_-\otimes T_+$ respectively (see \cite[Section 4.1.8]{federer} for the tensor product of currents) and that they satisfy
    $$F^*(\mathbb T_+)=d_p^2\mathbb T_+\quad\text{and}\quad F_*(\mathbb T_-)=d_p^2\mathbb T_-.$$
    In particular, they have H\"older-continuous super-potentials. Let $\Delta$ denote the diagonal of $X\times X$. Then $[\Delta]$ is a positive closed $(k,k)$-current on $X\times X$.

    \medskip

    Assume without loss of generality that $\|\varphi_j\|_q\le1$ for every $j$. Let $0\le j_0\le \kappa-1$ be an index such that $n:=n_{j_0+1}-n_{j_0}$ is maximal. Set $\widetilde G_j:=\varphi_j\circ\pi_2\circ F^{n_{j_0}-n_j}$ for $j\le j_0$ and $\widetilde G_j:=\varphi_j\circ\pi_1\circ F^{n_j-n_{j_0+1}}$ for $j\ge j_0+1$, where $\pi_1$ and $\pi_2$ are the canonical projections of $X\times X$ into its components. Define the auxiliary functions $\Phi^\pm$ on $X\times X$ by
    \begin{equation}\label{phipm}
        \Phi^\pm:=\Phi^\pm_{n_0,\dots,n_\kappa}=\sum_{j=0}^\kappa \Big((\kappa+1)\widetilde{G}_j+\frac{\kappa}{2}\widetilde{G}^2_j\Big)\pm \prod_{j=0}^\kappa \widetilde{G}_j.
    \end{equation}
    Using the fact that $\|\ddc \varphi_j\|_*\le\|\varphi_j\|_q\le1$, one can prove as in \cite[Lemma 3.2 and Lemma 3.4]{VW25Ax} that
    $$\|\ddc \Phi^\pm\wedge\mathbb{T}_+\|_*\le c_\kappa$$
    for some constant $c_\kappa$ independent of the $n_j$'s and the $\varphi_j$'s.

    Arguing as in the proof of Proposition \ref{pushforward_kahler}, we have $\|\varphi_j\circ f^l\|_{\log^q}\lesssim l^q$ for every $0\le j\le \kappa$ and $l\in\N^*$. Using the fact that the norm $\|\cdot\|_{\log^q}$ is submultiplicative (up to a constant), we get $\|\Phi^\pm\|_{\log^q}\lesssim n_\kappa^{q(\kappa+1)}$. We can then apply Corollary \ref{logq_sup_logq} and Proposition \ref{DS10JAG-3_4_2} to obtain that $\ddc\Phi^\pm\wedge\mathbb T_+$ has a $(C_\kappa n_\kappa^{q(\kappa+1)},q)$-$\log$-continuous super-potential for some constant $C_\kappa$ independent of the $n_j$'s and the $\varphi_j$'s.

    \medskip

    Finally, we would like to apply Theorem \ref{goal_kahler} to $F^{-1}$, $[\Delta]$ and $\Phi^\pm\wedge\mathbb T_+$ instead of $f$, $S$ and $\Phi$, then do the same calculations of the proof of \cite[Lemma 3.7]{VW25Ax}, and conclude by induction using the fact that with our choice of $n$ we have $n_k\le\kappa n$. Unfortunately, we do not know if $\Phi\wedge\mathbb T_+$ has the regularity needed to apply the theorem directly. Instead, we can do as in the proof of Theorem \ref{goal_kahler} itself to prove the desired estimate when $S$ is smooth. We then argue as in the proof of \cite[Theorem 1.2]{BD26AM} to extend the statement to the case $S=[\Delta]$. This concludes the proof.
\end{proof}

\section{Endomorphisms of $\mathbb{P}^k$}\label{endo}
The goal of this section is to prove Theorem \ref{logq_esp_test}. As we will see, the non-invertibility of $f$ will yield some difficulties in controlling the regularity of the pull-back operator, which was immediate in the invertible case (see Lemma \ref{pullback_holder}).

We take $X=\mathbb{P}^k$, $\omega=\omega_{FS}$ the Fubini-Study form, and $f$ an endomorphism of degree $d\ge2$ and maximal multiplicity $m$ at any point in $\mathbb P^k$. We define the operators $L_s:=\dfrac{1}{d^s}f^*$ and $\Lambda_s:=\dfrac{1}{d^s}f_*$ acting on $(s,s)$-currents and $(k-s,k-s)$-currents, respectively. They satisfy the following properties, see for instance \cite[Section 5]{DS09AM}.

\begin{lm}
    For every $0\le s\le k$, the operators $L_s$ and $\Lambda_s$ are continuous with respect to the $*$-topology. They also preserve the mass of positive currents. In particular, the dynamical degree $d_s(f)$ is equal to $d^s$.
\end{lm}

Following the conventions of Section \ref{prels} for super-potentials of currents, we normalize using $\omega_{FS}^s$, i.e., for every $0\le s\le k$ we choose $(\alpha_{s,1})=(\omega_{FS}^s)$. From now on, we also fix the constant $\kappa:=25k(4m)^{k+1}$.

\begin{lm} \label{pullback_holder}
    There exists a constant $C\ge1$, depending only on $f$, such that the operator $L_{k-s}$ is $(C,1/\kappa)$-H\"older-continuous on $\mathcal{D}_{k-s}^0\cap\mathcal{D}_{k-s}^1$ for every $s$.
\end{lm}

\begin{proof}
    We will prove the assertion in five steps.\\

    \textbf{Step 1}: $f_*(\omega_{FS})$ has a $1/(2m)$-H\"older-continuous super-potential.
    
    Indeed, by \cite[Lemma 5.4.8]{DS09AM} there exists a function $u$ which is $1/m$-H\"older-continuous and such that $\ddc u=f_*(\omega_{FS})-d^{k-1}\omega_{FS}$. Since we are normalizing potentials using $\fs$, we have $\mathcal{U}_{f_*(\omega_{FS})}=\mathcal{U}_{f_*(\omega_{FS})-d^{k-1}\fs}$. Then, we just need to apply Lemma \ref{hold_sup_hold} with $\Phi=u$.\\

    \textbf{Step 2}: $\big(f_*(\omega_{FS})\big)^{s+1}$ has a $\frac{1}{2^s(2m)^{s+1}}$-H\"older-continuous super-potential.
    
    To obtain Step 2 from Step 1, it suffices to apply repeatedly Proposition \ref{DS10JAG-3_4_2} to the current $f_*(\omega_{FS})$ and its powers.\\

    \textbf{Step 3}: there exists a constant $C_s>0$ such that, for every positive closed $(s+1,s+1)$-current $\Omega'$ with $\Omega'\le \omega_{FS}^{s+1}$, we have that $f_*(\Omega')$ has a $\big(C_s,\frac{1}{25k(4m)^{s+1}}\big)$-H\"older-continuous super-potential.
    
    Indeed, we have $f_*(\Omega')\le f_*(\omega_{FS}^{s+1})\le  f_*(\omega_{FS})^{s+1}$. Step 3 then follows from Step 2 by applying Proposition \ref{DNV18AM-1_1} with $S'=f_*(\Omega')$ and $S=f_*(\omega_{FS})^{s+1}$.\\

    \textbf{Step 4}: there exists a constant $C_s'>0$ such that, for every smooth $(s,s)$-form $\Omega$ with $\|\Omega\|_{\mathcal{C}^2}\le1$, we have that $\ddc\Lambda_{k-s}(\Omega)$ has a $\big(C_s',\frac{1}{25k(4m)^{s+1}}\big)$-H\"older-continuous super-potential.
    
    Indeed, since $\|\Omega\|_{\mathcal{C}^2}\le1$, we have $\|\ddc\Omega\|_{\mathcal{C}^0}\lesssim1$. Hence, we can write $\ddc\Omega=\Omega^+-\Omega^-$ with $\Omega^\pm$ positive closed currents in the same cohomology class and $\Omega^\pm\lesssim\omega_{FS}^{s+1}$. It follows that
    $$\ddc\Lambda_{k-s}(\Omega)=\frac{1}{d^{k-s}}f_*(\ddc\Omega)=\frac{f_*(\Omega^+)}{d^{k-s}}-\frac{f_*(\Omega^-)}{d^{k-s}}.$$
    By the definition of super-potentials, we have
    $$\mathcal{U}_{\ddc\Lambda_{k-s}(\Omega)}=\frac{1}{d^{k-s}}\mathcal{U}_{f_*(\Omega^+)}-\frac{1}{d^{k-s}}\mathcal{U}_{f_*(\Omega^-)}.$$
    Step 4 follows from Step 3 applied with $\Omega^\pm$ instead of $\Omega'$.\\

    \textbf{Step 5}: conclusion.

    Let $S$ be a current in $\mathcal{D}_{k-s}^0\cap\mathcal{D}_{k-s}^1$. Since $L_s$ is continuous for the $*$-topology, it is also continuous with respect to $\dist_2$. Therefore, we can assume $S$ smooth, and the general result will follow from an approximation with smooth forms. From the definition of $\|\cdot\|_{-2}$, it suffices to fix a smooth $(s,s)$-form $\Omega$ with $\|\Omega\|_{\mathcal{C}^2}\le 1$ and give an estimate of $|\langle L_{k-s}(S),\Omega\rangle|=|\langle S,\Lambda_{k-s}(\Omega)\rangle|$.

    Since $\{S\}=0$, we can write $S=\ddc U_S$ with $U_S$ smooth, so
    $$\langle S,\Lambda_{k-s}(\Omega)\rangle=\langle \ddc U_S,\Lambda_{k-s}(\Omega)\rangle=\langle U_S,\ddc\Lambda_{k-s}(\Omega)\rangle=\mathcal{U}_{\ddc\Lambda_{k-s}(\Omega)}(S).$$
    We conclude the proof applying Step 4.
\end{proof}

\begin{lm} \label{pushforward_doesthings}
    Let $R$ be a current in $\mathcal{D}^0_{k-s+1}$ whose super-potential $\mathcal{U}_R$ is $w$-continuous for some modulus of continuity $w$. Then we have
    $$\left|\mathcal{U}_{\frac{\Lambda_{s-1}^n(R)}{d^n}}(S)\right| \le C^\frac{\kappa}{\kappa-1}w\left(\|S\|_{-2}^{1/\kappa^n}\right)$$
    for all $n\ge0$ and for every $S\in\mathcal{D}_s^0\cap\mathcal{D}_s^1$, where $C$ is the constant given by Lemma \ref{pullback_holder}.
\end{lm}

\begin{proof}
    From Lemma \ref{pullback_holder} we have that
    $$\|L_s(S)\|_{-2}\le C\|S\|_{-2}^{\frac{1}{\kappa}}\le C^\frac{\kappa}{\kappa-1}\|S\|_{-2}^{\frac{1}{\kappa}}.$$
    Therefore, by induction we have that
    \begin{equation} \label{holderind}
        \|L_s^n(S)\|_{-2}\le C^\frac{\kappa}{\kappa-1}\|S\|_{-2}^{\frac{1}{\kappa^n}}.
    \end{equation}

    From Lemma \ref{pb-commute} we have that $\mathcal{U}_{\frac{\Lambda_{s-1}^n(R)}{d^n}}(S)=\mathcal{U}_R\big(L_s^n(S)\big)$. From the fact that $L_s$ preserves the mass of positive currents we have that $\|L_s^n(S)\|_*\le 1$. Therefore, we can apply the $w$-continuity of $\mathcal{U}_R$ together with \eqref{holderind} to obtain
    $$\left|\mathcal{U}_{\frac{\Lambda_{s-1}^n(R)}{d^n}}(S)\right|=C^{\frac{\kappa}{\kappa-1}}\left|\mathcal{U}_R\big(C^{\frac{\kappa}{1-\kappa}}L_s^n(S)\big)\right|\le C^{\frac{\kappa}{\kappa-1}} w\big(C^{\frac{\kappa}{1-\kappa}}\|L_s^n(S)\|_{-2}\big) \le C^{\frac{\kappa}{\kappa-1}}w\left(\|S\|_{-2}^{1/\kappa^n}\right).$$
    The proof is complete.
\end{proof}

\begin{cor} \label{pushforward_logq}
    Let $R$ be a current in $\mathcal{D}^0_{k-s+1}$ whose super-potential $\mathcal{U}_R$ is $(M,q)$-$\log$-continuous for some $M\ge1,q>0$. Then the super-potential $\mathcal{U}_{\Lambda_{s-1}^n(R)}$ of $\Lambda_{s-1}^n(R)$ is $(d^n\kappa^{nq}C_\kappa M,q)$-$\log$-continuous for every $n\ge1$, with $C_\kappa=C^{\frac{\kappa}{\kappa-1}}$.
\end{cor}

As a consequence of Corollary \ref{pushforward_logq}, we have the following result of convergence for currents with log-continuous super-potentials. The case of H\"older-continuous super-potentials is essentially contained in the proofs of \cite[Theorem 5.4.4 and Proposition 5.4.13]{DS09AM}, using a slightly different definition of super-potentials.

\begin{thm} \label{logq_esp_pot}
    For every $M\ge1,q>0$, and for all $R\in\mathcal{D}_{k-s+1}^0\cap\mathcal{D}_{k-s+1}^1$ with $(M,q)$-$\log$-continuous super-potential $\mathcal{U}_R$, we have
    \begin{equation}\label{lep}
        \left|\mathcal{U}_R\big(L_s^n(S)\big)\right| \lesssim \left(1+M^{\frac{1}{q+1}}\right)\left(\dfrac{\kappa}{d}\right)^{\frac{nq}{q+1}}
    \end{equation}
    for every $S\in\mathcal{D}_s^0\cap\mathcal{D}_s^1$, where the implicit constant depends only on $f$ and $q$.
\end{thm}

\begin{proof}
    From Lemmas \ref{pb-commute} and \ref{shuffle} we have
    $$\mathcal{U}_R\big(L_s^n(S)\big)=\mathcal{U}_{\frac{\Lambda_{s-1}^n(R)}{d^n}}(S)=\frac{1}{d^n}\mathcal{U}_{\Lambda_{s-1}^n(R)}(S)=\frac{1}{d^n}\mathcal{U}_S\big(\Lambda_{s-1}^n(R)\big).$$

    Applying Proposition \ref{skoda_clicks} to the current $\Lambda_{s-1}^n(R)$, and using Corollary \ref{pushforward_logq}, we obtain
    $$\left|\mathcal{U}_R\big(L_s^n(S)\big)\right|=\frac{1}{d^n}\left|\mathcal{U}_S\big(\Lambda_{s-1}^n(R)\big)\right| \le \frac{A}{d^n}+A(C_\kappa M)^{\frac{1}{q+1}} \left(\frac{\kappa}{d}\right)^{\frac{nq}{q+1}}.$$
    The proof is complete.
\end{proof}

Assume now that $\kappa=25k(4m)^{k+1}<d$. This is true, up to taking iterates, for a generic $f$, see for instance \cite[Lemma 5.4.5]{DS09AM}.

\smallskip

We are now ready to prove Theorem \ref{logq_esp_test}.

\begin{proof}[Proof of Theorem \ref{logq_esp_test}.]
    Assume without loss of generality that $\|\Phi\|_q\le1$. Define $R:=\ddc\Phi$, and observe that $\|R\|_*=\|\ddc\Phi\|_*\le\|\Phi\|_q\le1$. By Corollary \ref{logq_sup_logq} we have that $\mathcal{U}_R$ is $(M,q)$-$\log$-continuous for some constant $M\lesssim\|\Phi\|_{\log^q}\le\|\Phi\|_q\le1$. Again by Corollary \ref{logq_sup_logq}, we also have
    $$\langle L_s^n(S)-T^s,\Phi\rangle=\langle L_s^n(S-T^s),\Phi\rangle=\mathcal{U}_R\big(L_s^n(S-T^s)\big).$$
    It then suffices to apply Theorem \ref{logq_esp_pot} with $R$ and $(S-T^s)/2$ instead of $S$ to obtain \eqref{pk-end-1}. The proof is complete.
\end{proof}

\end{document}